\documentclass[12pt]{article}
\usepackage[hmargin={1.75truecm,2truecm},vmargin={2truecm,2truecm}]{geometry}
\usepackage[english]{babel}
\usepackage[utf8]{inputenc}
\usepackage{geometry}
\usepackage{multicol,float}
\usepackage{multirow}
\usepackage{hyperref}
\usepackage{amsmath}
\usepackage{amsthm}
\usepackage{amsfonts}
\usepackage{amssymb}
\usepackage{graphicx}
\usepackage{xcolor}
\usepackage{arydshln}
\usepackage[affil-it]{authblk}
\usepackage{natbib}

\usepackage{xspace}
\usepackage{booktabs}
\usepackage{subcaption}

\newtheorem{theorem}{Theorem}[section]
\newtheorem{example}[theorem]{Example}
\newtheorem{definition}[theorem]{Definition}

\newtheorem{remark}[theorem]{Remark}
\newtheorem{proposition}[theorem]{Proposition}
\newtheorem{lemma}[theorem]{Lemma}
\newtheorem{corollary}[theorem]{Corollary}

\usepackage{cleveref}
\Crefformat{appendix}{\S#2#1#3}
\crefname{theorem}{Theorem}{Theorems}
\crefname{example}{Example}{Examples}
\crefname{observation}{Observation}{Observations}
\crefname{remark}{Remark}{Remarks}
\crefname{proposition}{Proposition}{Propositions}
\crefname{lemma}{Lemma}{Lemmas}
\crefname{corollary}{Corollary}{Corollaries}
\crefname{algocf}{Algorithm}{Algorithms}	
\crefname{table}{Table}{Tables}	
\crefname{figure}{Figure}{Figures}
\crefname{algorithm}{Algorithm}{Algorithms}
\crefname{section}{Section}{Sections}
\crefname{algorithm}{Algorithm}{Algorithms}

\title{An efficient branch-and-cut approach for large-scale competitive facility location problems with limited choice rule}
\author[a]{Wei-Kun Chen}
\author[a]{Wei-Yang Zhang}
\author[a]{Yan-Ru Wang}
\author[d]{Shahin Gelareh}
\author[b]{Yu-Hong Dai}

\affil[a]{\small School of Mathematics and Statistics/Beijing Key Laboratory on MCAACI, Beijing Institute of Technology, Beijing 100081, China}
\affil[b]{\small Academy of Mathematics and Systems Science, Chinese Academy of Sciences, Beijing 100190, China}
\affil[c]{\small School of Mathematical Sciences, University of Chinese Academy of Sciences, Beijing 100049, China}
\affil[d]{\small D\'{e}partement R\&T, IUT de B\'{e}thune, Universit\'{e} d'Artois, B\'{e}thune F-62000, France}
\date{}

\DeclareMathOperator{\conv}{conv}

\newcommand{\CFLP}{\text{CFLP}\xspace}
\newcommand{\CFLPLR}{\text{CFLPLCR}\xspace}
\newcommand{\MCFLRU}{{MCFLRU}\xspace}

\newcommand{\LP}{{LP}\xspace}
\newcommand{\MIP}{{MIP}\xspace}
\newcommand{\MILP}{{MILP}\xspace}
\newcommand{\MINLP}{{MINLP}\xspace}
\newcommand{\MICQP}{{MICQP}\xspace}

\newcommand{\BnC}{\text{B\&C}\xspace}
\newcommand{\GBD}{\text{GBD}\xspace}

\newcommand{\solver}[1]{\textsc{#1}\xspace}

\newcommand{\Gurobi}{\solver{Gurobi}}

\newcommand{\CF}{\mathcal{F}}
\newcommand{\CN}{\mathcal{N}}

\newcommand{\CP}{\mathcal{P}}
\newcommand{\CQ}{\mathcal{Q}}
\newcommand{\CS}{\mathcal{S}}
\newcommand{\CT}{\mathcal{T}}
\newcommand{\X}{\mathcal{X}}
\newcommand{\Y}{\mathcal{Y}}
\newcommand{\Z}{\mathbb{Z}}
\newcommand{\R}{\mathbb{R}}
\newcommand{\LIN}{\text{L}\xspace}

\newcommand{\SI}{\text{SI}\xspace}
\newcommand{\SIbar}{\text{$\overline{\text{SI}}$}\xspace}

\newcommand{\order}{\rm{Ord}}

\newcommand{\DATAONE}{T1\xspace}
\newcommand{\DATATWO}{T2\xspace}

\newcommand{\testGBD}{\texttt{GBD}\xspace}
\newcommand{\testSI}{\texttt{B\&C+SI}\xspace}
\newcommand{\testLSI}{\texttt{B\&C+LSI}\xspace}

\newcommand{\nh}{\text{NH}\xspace}

\newcommand{\tblm}{$m$}
\newcommand{\tbln}{$n$}
\newcommand{\tblgamma}{$\gamma$}
\newcommand{\tblgapclosed}{\texttt{GI(\%)}\xspace}
\newcommand{\tblT}{\texttt{T}\xspace}
\newcommand{\tblTg}{\texttt{T\,(G\%)}\xspace}
\newcommand{\tblTlp}{$\texttt{T}^{\texttt{1}}$\xspace}
\newcommand{\tblTC}{$\texttt{CT}$\xspace}
\newcommand{\tblG}{\texttt{G\%}\xspace}
\newcommand{\tblN}{\texttt{N}\xspace}
\newcommand{\tblZ}{\texttt{$\nu$}\xspace}
\newcommand{\tblZUB}{\text{$\texttt{UB}^1$}\xspace}
\newcommand{\tblZLB}{\text{$\texttt{LB}^1$}\xspace}
\newcommand{\ave}{Ave.}
\newcommand{\tblSolved}{\texttt{Sol.}}

\newcommand{\UB}{\texttt{UB}}
\newcommand{\LB}{\texttt{LB}}

\newcommand{\RHS}{f}

\begin{document}
	\maketitle

\begin{abstract}
	We consider the competitive facility location problem with limited choice rule (\CFLPLR), which attempts to open a subset of facilities to maximize the net profit of a ``newcomer'' company, requiring customers to patronize only a limited number of opening facilities and an outside option. 
	We investigate the polyhedral structure of a mixed 0-1 set, defined by the function characterizing the probability of a customer patronizing the company's open facilities, and propose an efficient branch-and-cut (\BnC) approach for the \CFLPLR based on newly proposed mixed integer linear programming (\MILP) formulations. 
	Specifically, by establishing the submodularity of the probability function, we develop an \MILP formulation for the \CFLPLR using the submodular inequalities.
	For the special case where each customer patronizes at most one open facility and the outside option, we show that the  submodular inequalities can characterize the convex hull of the considered set and provide a compact \MILP formulation.
	{Moreover}, for the general case, we strengthen the submodular inequalities by sequential lifting, resulting in a class of facet-defining inequalities.
	The proposed lifted submodular inequalities are shown to be stronger than the classic submodular inequalities, enabling to obtain another \MILP formulation with a tighter linear programming (\LP)  relaxation.
	By extensive numerical experiments, we show that thanks to the tight \LP relaxation, the proposed \BnC approach outperforms the state-of-the-art generalized Benders decomposition approach by at least one order of magnitude. 
	Furthermore, it enables to solve \CFLPLR instances with $10000$ customers and $2000$ facilities.
\end{abstract}

\section{Introduction}\label{sect:Intro}

In the competitive facility location problem (\CFLP), a ``newcomer'' company attempts to enter a shared market, where competitors' facilities already exist, by opening new facilities to attract the customers' demand or buying power.
Each customer (or customer zone) is partially or totally served by the facilities located by the companies, {according to} some prespecified rules (also called \emph{customer choice rules}). 
The objective of the newcomer company is to maximize the market share or revenue. 
\CFLP arises in a wide variety of applications such as   
designing preventive healthcare networks  \citep{Zhang2012,Krohn2021}, 
locating park-and-ride facilities \citep{Aros-Vera2013,Freire2016},
building charging stations for electric vehicles \citep{Zhao2020}, 
deploying lockers for last-mile delivery \citep{Lin2022}, and
opening new retail stores \citep{Mendez-Vogel2023a}.
We refer to the surveys 
\citet{Eiselt1997,Plastria2001,Suarez-Vega2004,Kress2012,Drezner2024} for more applications and detailed reviews of \CFLP{s}.

In general, the patronizing behavior of customers is based on the attractiveness/utility of open facilities, which is predetermined by a set of attributes such as distance, transportation costs, waiting times, and facility capacities.
In the literature, several customer choice rules have been  proposed to model the patronizing behavior of customers.
Specifically, the \emph{proportional choice rule}, which includes the well-known multinomial logit \citep{McFadden1974} %
and Huff-based gravity \citep{Huff1964} rules as special cases, suggests that customers distribute their buying power among all open facilities by the newcomer company and an outside option (which corresponds to the aggregated facility of all existing competitors' facilities) in proportion to their {utilities}. 
The proportional choice rule has been widely applied in various context; see  \cite{Benati2002,Aboolian2007,Berman2009a,Kucukaydin2011a,Aros-Vera2013,Haase2014,Ljubic2018,Drezner2018,Mai2020} among many of them.
In contrast, the \emph{binary choice rule} assumes that each customer patronizes the facility with the highest utility \citep{Beresnev2013,Fernandez2017,Lancinskas2020}.
An intermediate approach, known as the \emph{partially binary choice rule}, is similar to the proportional choice rule but requires  the buying power of a customer to be distributed among the newcomer company's open facility with the highest utility and an outside option \citep{Suarez-Vega2004,Biesinger2016,Fernandez2017}.
\citet{Lin2021} further {generalized} the proportional choice rule and the partially binary choice rule by proposing the \emph{limited choice rule}, 
which assumes that customers first form a consideration set by choosing no more than a predetermined number of open facilities in descending order of utilities, 
and then split their buying power among the facilities in the consideration set and an outside option proportional to their utilities.
Let $n$ and $m$ denote the number of potential facilities considered by the newcomer company and the number of customers, respectively.
Mathematically, the probability of customer $i$ (or the proportion of its buying power) attracted by the open facilities of the newcomer company under the limited choice rule can be presented as 
\begin{equation}\label{setfunction}
	\Phi_i(\CS) = \max \left\{\frac{ \sum_{j \in \CS'} u_{ij}}{\sum_{j \in \CS'} u_{ij} + u_{i0}} : \CS' \subseteq \CS, \ |\CS'|\le \gamma_i\right\}.
\end{equation}
Here, $\CS \subseteq [n]:=\{1, 2, \ldots, n  \}$ is the subset of open facilities chosen by the newcomer company, 
$u_{ij}\geq 0$ is the utility of facility $j \in [n]$ to customer $i \in [m]:=\{1, 2, \ldots, m\}$,
$u_{i0}\geq 0$ is the utility of the outside option to customer $i$, 
and $\gamma_i$ is the restriction on the size of the consideration set for customer $i$.
When $\gamma_i \geq n$, the limited choice rule reduces to the proportional choice rule; when $\gamma_i =1$, it reduces to the partially  binary choice rule.
For other customer choice rules, we refer to \citet{Lancinskas2017,Gentile2018,Fernandez2021,Mendez-Vogel2023a,Lin2024a} and references therein.

In this paper, we consider the \emph{competitive facility location problem with limited choice rule} (\CFLPLR).
The considered problem attempts to locate new facilities in the competitive environment while maximizing the net profit of the company, which equals to the revenue collected by the open facilities minus the fixed cost of these facilities, and 
can be formally stated as
\begin{equation}\label{prob:original-max-S}\tag{\CFLPLR}
	\max_{\CS \in {\CF}}  \sum_{i \in [m]} b_i \Phi_i(\CS) - \sum_{j \in \CS} f_j,
\end{equation}
where $b_i$ is the buying power of customer $i \in [m]$, $f_j$ is the fixed cost of opening facility $j \in [n]$, and $\CF$ consists of (feasible) subsets of $[n]$, determined by the newcomer company.
For simplicity of presentation, we assume that $\CF = 2^{[n]}$ (i.e., $\CF$ includes all subsets of $[n]$), which is the case considered in \cite{Lin2021} (also called \CFLPLR in the paper).
However, our results in this paper can also be applied to other variants that consider other practical constraints on $\CF$ such as cardinality constraints, sizes of the facilities, and connectivity of the facilities. %
We refer to \cite{Ljubic2018} and the references therein for a detailed discussion on these practical constraints.

\subsection{Literature review}

A widely investigated case of \eqref{prob:original-max-S} is the \emph{maximum capture facility location problem with random utilities} (\MCFLRU) where $f_j = 0$ for all $j \in [n]$, $\gamma_i \ge n$ for all $i \in [m]$, and $\CF = \{\CS \subseteq [n]: |\CS| = p\}$ ($p\in \mathbb{Z}_+$).
\citet{Benati2002} formulated the \MCFLRU as a multiple-ratio 0-1 fractional programming problem with $n$ binary location variables and proposed the first mixed integer linear programming (\MILP) reformulation based approach to solve it.
Subsequently, \citet{Haase2009,Aros-Vera2013,Zhang2012} proposed alternative \MILP reformulations;
\citet{Haase2014} provided a computational comparison of these \MILP reformulations and concluded that the one from \citet{Haase2009} {offers} the best performance on a testset of random instances.
\citet{Freire2016} developed a strengthened version of the \MILP reformulation that provides a tighter LP relaxation than the one in \citet{Haase2009}.
Although the above \MILP formulations enable to globally solve \MCFLRU{s} by off-the-shelf commercial solvers, the huge problem size makes them  challenging to be solved efficiently, especially for large-scale cases (as the linearization of the multiple-ratio fractional objective function needs to introduce {$\mathcal{O}(mn)$ or $\mathcal{O}(mn^2)$} variables and constraints \citep{Haase2014}).
\citet{Altekin2021} developed a mixed-integer conic quadratic programming (\MICQP) reformulation for the \MCFLRU that involves only $m+n$ variables. 
Computational results on their testsets demonstrate the efficiency of solving the \MICQP formulation over solving the above \MILP formulations  using off-the-shelf commercial  solvers.
In another line of research, 
\citet{Ljubic2018} employed the submodularity of the probability function $\Phi_i(\CS)$ \citep{Benati2002}, and proposed a branch-and-cut (\BnC) algorithm based on the well-known submodular inequalities \citep{Nemhauser1981} for solving \eqref{prob:original-max-S} with proportional choice rule (i.e.,  $\gamma_i \geq n$ for all $i \in [m]$). 
The \BnC algorithm also involves $m+n$ variables and is particularly suitable to be applied to solving large-scale \CFLPLR{s} with proportional choice rule.
\citet{Ljubic2018} and \citet{Mai2020} proposed similar efficient \BnC algorithms but based on outer-approximation inequalities.

In contrast to the special case $\gamma_i\geq n$ for all $i \in [m]$, where only $n$ location variables are involved in the multiple-ratio 0-1 fractional programming problem, for  \eqref{prob:original-max-S} with arbitrarily positive {integers} $\gamma_i$, additional $mn$ allocation variables are needed for characterizing the mixed integer nonlinear programming (\MINLP) formulation \citep{Lin2021}, making it challenging for the algorithmic design.
\citet{Lin2021} proposed a generalized Benders decomposition (\GBD) algorithm \citep{Geoffrion1972} that projects the $mn$ allocation variables out from the \MINLP formulation and separates the so-called Benders cuts on the fly. 
Computational results show that the \GBD algorithm outperforms an outer-approximation algorithm and an \MICQP reformulation based approach.
Nevertheless, due to the weakness of the Benders cuts in the context of solving \CFLPLR{s}, the search tree by the \GBD algorithm is usually very large \citep{Lin2021}, prohibiting it from solving large-scale \CFLPLR{s}. 
It is worthwhile noting that this drawback comes from the derivation of the Benders cuts in the \GBD algorithm where the integrality constraints on the binary variables are completely neglected \citep{Bodur2017,Rahmaniani2017}, resulting in weak cuts.

\subsection{Contributions}

The goal of this paper is to fill this research gap by proposing an efficient \BnC approach that is capable of solving \emph{large-scale} \CFLPLR{s}.
The proposed \BnC approach is based on new strong valid inequalities derived from a mixed 0-1 set, which closely relates to the hypograph of the probability function $\Phi_i(\CS)$ but takes the integrality constraints into consideration. 
In particular, 
\begin{itemize}
	\item We establish the submodularity of the probability function $\Phi_i(\CS)$, which enables to characterize the mixed 0-1 set using the submodular inequalities, thereby obtaining an \MILP formulation for \eqref{prob:original-max-S}.
	For the special case where the customer behavior follows the partially binary rule, we show that a linear number of submodular inequalities are sufficient to describe the convex hull of the mixed 0-1 set.
	Using this result, we develop a compact \MILP formulation with a tighter  linear programming (\LP) relaxation than the Benders reformulation of the \CFLPLR in \cite{Lin2021}.
	\item We strengthen the submodular inequalities by sequential lifting \citep{Wolsey1976a,Richard2011}, obtaining a class of facet-defining inequalities, called lifted submodular inequalities.
	As a byproduct of analysis, we identify a necessary condition for an inequality to be facet-defining for the  convex hull of a generic 	submodular set.
	With this result, the lifted submodular inequalities are shown to be stronger than the classic submodular inequalities, leading to another \MILP formulation for the \eqref{prob:original-max-S} with a tighter \LP relaxation.
	We also discuss the algorithmic design for the computation of the lifted submodular inequalities.
\end{itemize}

The new developed \MILP formulations are solved by the proposed \BnC approach in which the submodular or lifted submodular inequalities are separated on the fly. 
Our extensive computational results show that thanks to the tight \LP relaxation, the proposed \BnC approach outperforms the state-of-the-art \GBD approach in \cite{Lin2021} by at least one order of magnitude.
Moreover, it enables to solve \CFLPLR instances with up to $10000$ customers and $2000$ facilities.

\subsection{Organization and Notations}

The rest of the paper is organized as follows.
\cref{sect:Prob} reviews the \MINLP formulation  of \eqref{prob:original-max-S} and discusses the weakness of the state-of-the-art \GBD approach of \cite{Lin2021} for solving it.
\cref{sect:SI} establishes the submodularity of the probability function $\Phi_i(\CS)$ and derives an \MILP formulation for \eqref{prob:original-max-S} based on the submodular inequalities.
\cref{sect:LSI} develops the lifted submodular inequalities utilizing  sequential lifting and discusses the algorithmic design for their computation.
\cref{sect:Imple} presents the implementation details of the proposed \BnC algorithm.
\cref{sect:Compu} reports the computational results.
Finally, \cref{sect:Conclusion} draws the conclusion.

Throughout the paper, for any $n \in \Z_+$, we denote {$[n] = \{1, 2, \dots, n\}$} and $[n] = \varnothing$ if $n = 0$.
Let $\boldsymbol{0}$ and $\boldsymbol{e}_j$ denote all zeros $n$-dimensional vector and the $j$-th $n$-dimensional unit vector, respectively.
For a subset $\CS \subseteq [n]$, we denote the characteristic vector of $\CS$ by $x^{\CS} $, that is, $x^{\CS}_j = 1$ if $j \in \CS$ and $x^{\CS}_j = 0$ otherwise. 
Conversely, for a vector $x \in \{0,1\}^n$, we denote the support of $x$ by $\CS^{x} = \{j \in [n]: x_j = 1\}$.
Given a vector $u \in \mathbb{R}^n$ and a subset $\CS \subseteq [n]$, we denote by $u(\CS)=\sum_{j \in \CS} u_j$.
{For} simplicity, we abbreviate $\CS \cup \{j\}$ and $\CS \backslash \{j\}$ as $\CS\cup j$ and $\CS\backslash j$ where $\CS\subseteq [n]$ and $j \in [n]$.
For a real value $a$, we denote ${(a)}^+= \max\{a,0\}$.

\section{Problem formulation}
\label{sect:Prob}

In this section, we first present the mathematical formulation for \eqref{prob:original-max-S}.
Then we briefly review the \GBD algorithm of \citet{Lin2021} and discuss its weakness that motivates the development of our proposed approach.

\subsection{An \MINLP formulation for \eqref{prob:original-max-S}}

Let $x \in \{0,1\}^{n}$ be location variables such that $x_j = 1$ if facility $j$ is open and $x_j = 0$ otherwise.
{Let $y \in \{0,1\}^{m\times n}$ be allocation variables such that $y_{ij} = 1$ if customer $i$ patronizes facility $j$ and $y_{ij} = 0$ otherwise.} 
Then \eqref{prob:original-max-S} can be formulated as an \MINLP problem \citep{Lin2021}:
\begin{subequations}\label{prob:original-max}
	\begin{align}
		\max_{x,\,y} \ & \sum_{i\in [m]}b_i\frac{\sum_{j \in [n]}u_{ij}y_{ij}}{\sum_{j \in [n]}u_{ij}y_{ij}+u_{i0}}-\sum_{j \in [n]}f_jx_j
		\label{objective}\\
		\text{s.t.} \   &y_{ij} \le x_j, \ \forall \ i \in [m], \ j \in [n], 
		\label{cons:xy}\\
		& \sum_{j \in [n]}y_{ij} \le \gamma_i, \ \forall \ i \in [m], 
		\label{cons:sum-y}\\
		& x_j \in \{0,1\}, \ \forall \ j \in [n], 
		\label{cons:x-bin}\\
		& y_{ij} \in \{0,1\}, \ \forall \ i \in [m], \ j \in [n] \label{cons:y-bin}.
	\end{align}
\end{subequations}
The objective function $\eqref{objective}$ maximizes the net profit, where the first term is the revenue collected by the open facilities and the second term is the fixed cost of opening the facilities.
{Constraints $\eqref{cons:xy}$ indicate that facility $j$ is considered by customer $i$ only if it is open.}
Constraints $\eqref{cons:sum-y}$ state that the number of facilities considered by customer $i$ is upper bounded by $\gamma_i$.
Finally, constraints \eqref{cons:x-bin} and \eqref{cons:y-bin} restrict the decision variables to be binary.

Note that the integrality constraints \eqref{cons:y-bin} in problem \eqref{prob:original-max} can be equivalently relaxed.
Indeed, let (\ref{prob:original-max}') be the problem in which \eqref{cons:y-bin} is replaced with 
\begin{equation}\label{cons:y-cont}
	y_{ij} \ge 0, \ \forall \ i \in [m], \ j \in [n], \tag{\ref{cons:y-bin}'}
\end{equation}
in problem \eqref{prob:original-max}.
By noting that there must exist an optimal solution $(x^*, y^\ast)$ of problem (\ref{prob:original-max}') such that for each $i \in [m]$, $y^\ast_{ij} = 1$ holds for the $\min\{\gamma_i, \sum_{j \in [n]} x_j^\ast\}$ open facilities with the highest utilities $u_{ij}$,  problems \eqref{prob:original-max} and (\ref{prob:original-max}') are equivalent.

\subsection{Weakness of the \GBD algorithm of \cite{Lin2021}} \label{subsect:weakness-GBD}

Note that due to the huge number of the allocation variables $y$ and the related constraints (especially when $m$ and $n$ are large), the \MINLP problem \eqref{prob:original-max} is generally difficult to solve. 
To tackle this difficulty, \citet{Lin2021} projected variables $y$ out from the formulation
and reformulated problem \eqref{prob:original-max} as
\begin{equation}\label{prob:GBD-pre}
\max_{x,\,w}  \left\{\sum_{i \in [m]} b_i w_i - \sum_{j \in [n]} f_j x_j : w_i \le \phi_i(x), \ \forall \ i \in [m], ~
x_j \in \{0,1\}, \ \forall \ j \in [n] \right\},
\end{equation}
where for each $i \in [m]$, $w_i$ is a continuous variable representing the probability of customer $i$ patronizing the newcomer company's open facilities, {and}
\begin{equation}\label{def:phi-i}
	\phi_i(x) = \max_{y} \  \left\{ \frac{\sum_{j \in [n]} u_{ij} y_{ij}}{\sum_{j \in [n]} u_{ij}  y_{ij} + u_{i0}} ~:~
	\sum_{j \in [n]} y_{ij} \le \gamma_i, \ 
	y_{ij} \le x_j, ~
	y_{ij} \in [0,1], ~\forall \ j \in [n]\right\}.
\end{equation}
Note that $ \phi_i(x^{\CS})=\Phi_i(\CS) $ holds for all $\CS \subseteq [n]$, where $\Phi_i(\CS) $ is defined in \eqref{setfunction}, and $x^{\CS}_j = 1$ if $j \in \CS$ and $x^{\CS}_j = 0$ otherwise. 
Thus, in the subsequent discussions, we use $\phi_i$ and $\Phi_i$ interchangeably.
\cite{Lin2021} derived a Benders reformulation of problem \eqref{prob:GBD-pre} in which $w_{i} \leq \phi_i(x)$, $i\in [m]$, are replaced by the so-called Benders cuts, and solved the reformulation using a \GBD algorithm \citep{Geoffrion1972}.  
In the following, we shall briefly discuss the Benders cuts; see \cite{Lin2021} for a detailed discussion on the derivation.

To present the Benders cuts, we first, 
for simplicity of presentation, drop subscript $i$ and abbreviate $w_i \leq \phi_i(x)$ as $w \leq \phi(x)$ where 
\begin{equation}\label{phidef}
	\phi(x) = \max_y \  \left\{ \frac{\sum_{j \in [n]} u_{j} y_{j}}{\sum_{j \in [n]} u_{j}  y_{j} + u_{0}} ~:~
	\sum_{j \in [n]} y_{j} \le \gamma, \ 
	y_{j} \le x_j,  \
	y_{j} \in [0,1], ~\forall \ j \in [n]\right\}.
\end{equation}
 Throughout the paper, we assume, without loss of generality,  that {$u_0 > 0$}, $u_1 \ge u_2 \ge \cdots \ge u_n \ge u_{n+1} := 0$, and $\gamma \ge 1$.
Given a point $\bar{x} \in \{0,1\}^n$, let %
$\{ j \in [n]\, : \, \bar{x}_j =1\}= \{j_1, j_2, \ldots, j_s\}$ with $j_1 < j_2 < \cdots < j_s$. 
Then the Benders cut in \cite{Lin2021} can be presented as follows:
\begin{equation}\label{ineq:GBD}
	w \le \phi(\bar{x}) + \sum_{j \in [n]} \lambda_j (x_j  - \bar{x}_j)
\end{equation}
where for $j \in [n]$, if  $s\leq \gamma$, $\lambda_j = u_0 u_j / (\sum_{\ell \in \{j_1, \dots, j_s\}} u_{\ell} +u_0)^2$ and $\lambda_j = u_0 (u_j - u_{j_{\gamma + 1}})^+ / (\sum_{\ell \in \{j_1, \dots, j_{\gamma}\}} u_{\ell}+u_0)^2$ otherwise{; and 
$(a)^+ = \max\{a, 0\}$ for $a \in \mathbb{R}$}.
It should be mentioned that (i) the Benders cuts are valid for the hypograph of function $\phi(x)$: 
$$\X_{\LIN}= \{ (w, x) \, : \, w \leq \phi(x) \}$$ and (ii) those defined by integral points $\bar{x} \in \{0,1\}^n$ are sufficient to provide a linear description of 
\begin{equation}\label{def:X}
	\X : = \X_{\LIN}\cap \{ (w,x)\, : \, x \in \{0,1\}^n \} =  \left\{(w,x) \in \R \times \{0,1\}^n: w \le \phi(x)\right\}.
\end{equation}

Although the \GBD approach of \cite{Lin2021} can overcome the disadvantage of the large problem size in \eqref{prob:original-max}, it could, however, still be inefficient due to \emph{the weakness of the Benders cuts}.
An example for this is presented in \cref{exp:GBD-weak}.
The reason behind this is that the derivation of the Benders cuts completely neglects the integrality constraints on variables $x$ \citep{Bodur2017,Rahmaniani2017};
in other words, the Benders cuts are not only valid for $\X$ but also valid for its {relaxation $\X_{\LIN}$}.
The weakness of the Benders cuts leads to a weak \LP relaxation of the Benders reformulation, thereby resulting in a large \BnC search tree by the \GBD algorithm, as will be demonstrated in \cref{sect:Compu}.

\begin{example}\label{exp:GBD-weak}
	Consider an example of $w \leq \phi(x)$ where 
	\begin{equation*}
		\phi(x) = \max \left\{ \frac{5 y_1 + 4 y_2 + 3 y_3}{5 y_1 + 4 y_2 + 3 y_3 + 10}: 
		y_1 + y_2 + y_3 \le 1, \  
		y_j \le x_j,  \ 
		y_j \in \{0,1\}, j = 1, 2, 3
		\right\}.
	\end{equation*}
	By simple computation, the Benders cut \eqref{ineq:GBD} induced by point $\bar{x} = (0, 0, 0)$ can be written as
	\begin{equation}\label{ineq:example-GBD-1}
		w  \le  \frac{5}{10} x_{1} + \frac{4}{10} x_{2} + \frac{3}{10} x_{3}.
	\end{equation}
	Notice that inequality \eqref{ineq:example-GBD-1} is a valid inequality for the mixed 0-1 set $\X^1 =\{ (w,x) \in \mathbb{R}\times\{0,1\}^{3} \, : \, w \leq \phi(x) \}$.
	However, it is very weak; indeed, the face defined by it reads $\CF^1  =  \{ (w, x) \in \X^1 : w  = \frac{5}{10} x_{1} + \frac{4}{10} x_{2} + \frac{3}{10} x_{3} \} = \{(0,0,0,0)\}$, and thus its dimension is only $0$. 
	Later, we will see that inequality \eqref{ineq:example-GBD-1} can be strengthened into 
	\begin{equation}\label{ineq:example-SM-1}
		w  \le  \frac{5}{15} x_{1} + \frac{4}{14} x_{2} + \frac{3}{13} x_{3},
	\end{equation}
	which defines a (strong) facet-defining inequality of $\conv(\X^1)$.
\end{example}

In this paper, we will bypass the weakness of the \GBD algorithm by developing a new efficient \BnC approach based on a polyhedral study of the mixed 0-1 set $\X$, which simultaneously takes the hypograph of function $\phi(x)$ and the integrality constraints $x\in \{0,1\}^n$ into consideration.
We will derive strong valid inequalities for $\X$ that as Benders cuts, are able to characterize $\X$.

\section{An \MILP formulation based on submodular inequalities}
\label{sect:SI}

In this section, we first show the submodularity of function $\Phi(\CS)$,
\begin{equation}\label{def:phiS1}
	\Phi(\CS) = \max\left\{
	\frac{u(\CS')}{u(\CS') + u_0} :  \CS' \subseteq \CS, \ |\CS'| \le \gamma
	\right\},
\end{equation}
obtained by dropping subscript $i\in [m]$ from \eqref{setfunction}.
Using this favorable property and the results in \cite{Nemhauser1981}, we  can derive a linear formulation for $\X$ using the submodular inequalities, and thus reformulate \eqref{prob:original-max-S} as an \MILP problem.
We illustrate the strength of submodular inequalities for the special case $\gamma = 1$.

\subsection{Submodularity of $\Phi(\CS)$} \label{subsect:submodularity}

We start with the definition of the submodular function, which can be found in, e.g., \citet{Nemhauser1978}.
\begin{definition}\label{def:submodular}
	A real-valued function $\Psi: 2^{[n]} \rightarrow \R$ is submodular if $\Psi(\CS \cup j) - \Psi(\CS) \ge \Psi(\CT \cup j) - \Psi(\CT)$ for all $\CS \subseteq \CT \subseteq [n]$ and $j \in [n] \backslash \CT$.
\end{definition}
\noindent Intuitively, if the marginal gain $\Psi(\CS\cup j)-\Psi(\CS)$ of adding an element $j$ diminishes with the size of the set, then function $\Psi(\CS )$ is submodular.
For the function $\Phi(\CS)$ in \eqref{def:phiS1}, 
if $\gamma \geq  n$,  it reduces to $\frac{u(\CS)}{u(\CS) + u_0}$. 
For this special case, 
it is known that $\Phi(\CS)$ is a submodular function; see  \cite{Berman1998,Benati2002,Ljubic2018}.
The following theorem shows, somewhat surprisingly, that this statement also holds for an arbitrarily positive integer $\gamma$. 
\begin{theorem}\label{thm:submodularity}
	Function $\Phi(\CS)$ is  submodular for any $\gamma \in \mathbb{Z}_+$ with $\gamma\geq 1$.
\end{theorem}
To establish the result in \cref{thm:submodularity}, we need the following notations 
\begin{equation}\label{def:Sgamma}
	k_\CS = \left\{
	\begin{aligned}
		& n+1, && \text{ if } s\leq \gamma-1;\\
		&j_{\gamma}, && \text{otherwise},
	\end{aligned}
	\right.
	\text{ and } \CS^\gamma = \CS \cap [k_\CS],
\end{equation}
where $\CS = \{j_1, j_2, \dots, j_s\} \subseteq [n]$ with $j_1 < j_2 < \cdots < j_s$.
From the definitions of $k_{\CS}$, $\CS^{\gamma}$, and $\Phi(\CS)$,
we immediately obtain the following two lemmas.
\begin{lemma}\label{lem:subset}
	For $\CS\subseteq \CT\subseteq [n]$, it follows (i) $k_\CS \geq k_{\CT}$ and $u_{k_\CS} \leq u_{k_\CT}$ and (ii) $u(\CT^\gamma) \geq u(\CS^\gamma)$.
\end{lemma}

\begin{lemma}\label{g_phi_lemma}
	For $\CS \subseteq [n]$, let 
	\begin{equation}
		g(\CS) =  \frac{u(\CS)}{u(\CS) + u_0}.
	\end{equation}
	Then $\Phi(\CS) = \Phi(\CS^\gamma) = g(\CS^\gamma)$.
\end{lemma}

For $\CS = \{j_1, j_2, \dots, j_s\} \subsetneq [n]$ with $j_1 <j_2 < \cdots < j_s$ and $j \in [n]\backslash \CS$, we have
\begin{equation*}
	{(\CS\cup j)}^\gamma = \left\{
	\begin{aligned}
		& \CS^\gamma\cup j, && \text{ if } s \leq \gamma-1;\\
		& \CS^\gamma, && \text{ if } s \geq \gamma~\text{and}~j > k_{\CS};\\
		& \CS^\gamma \cup j \backslash k_{\CS},  && \text{ if } s \geq \gamma~\text{and}~ j < k_{\CS}.
	\end{aligned}
	\right.
\end{equation*}
In all {three} cases, we have {$u((\CS\cup j)^\gamma)=u(\CS^\gamma)+(u_j - u_{k_\CS})^{+}$}.
Let $\rho_j(\CS) = \Phi(\CS\cup j) -\Phi(\CS)$. Then by Lemma \ref{g_phi_lemma}, we have 
\begin{equation}\label{rho-ge-gamma-2}
	\begin{aligned}
		\rho_j(\CS) 
		& =\Phi(\CS\cup j) -\Phi(\CS)= g((\CS \cup j)^\gamma) - g(\CS^\gamma) \\
		& =  \frac{u((\CS\cup j)^{\gamma})}{u((\CS\cup j)^{\gamma}) + u_0} - \frac{u(\CS^{\gamma})}{u(\CS^{\gamma}) + u_0}= \frac{u_0(u_j - u_{k_\CS})^{+}}{(u(\CS^{\gamma}) + (u_j - u_{k_\CS})^{+} + u_0) \cdot (u(\CS^{\gamma}) + u_0)}.
	\end{aligned}
\end{equation}
\begin{proof}[Proof of \cref{thm:submodularity}]
	For any $\CS\subseteq \CT \subseteq [n]$, we have
	\begin{equation*}
		\begin{aligned}
			\rho_j(\CS) 
			& = \frac{u_0(u_j - u_{k_\CS})^{+}}{(u(\CS^{\gamma}) + (u_j - u_{k_\CS})^{+} + u_0) \cdot (u(\CS^{\gamma}) + u_0)} \\
			& \overset{(a)}{\geq}  \frac{u_0(u_j - u_{k_\CS})^{+}}{(u(\CT^{\gamma}) + (u_j - u_{k_\CS})^{+} + u_0) \cdot (u(\CT^{\gamma}) + u_0)}\\
			& = \frac{(u_j - u_{k_\CS})^{+}}{u(\CT^{\gamma}) + (u_j - u_{k_\CS})^{+} + u_0} \cdot \frac{u_0}{u(\CT^{\gamma}) + u_0}\\
			& \overset{(b)}{\geq} \frac{(u_j - u_{k_\CT})^{+}}{u(\CT^{\gamma}) + (u_j - u_{k_\CT})^{+} + u_0} \cdot \frac{u_0}{u(\CT^{\gamma}) + u_0}= \rho_j(\CT),
		\end{aligned}
	\end{equation*}
	where (a) follows from \cref{lem:subset} (ii) and (b) follows from $(u_j - u_{k_\CS})^{+} \geq (u_j - u_{k_\CT})^{+}$ (from \cref{lem:subset} (i)) and $\frac{b}{a+b} \geq \frac{c}{a+c}$ for all $a,b,c \in \mathbb{R}_+$ with $b \geq c$.
\end{proof}

\subsection{An \MILP reformulation for \eqref{prob:original-max-S}}
\label{subsect:SC-ineq}

Using the submodularity of $\Phi(\CS)$ in \cref{thm:submodularity} and the results in \citet{Nemhauser1981}, we know that either of the following two families of submodular inequalities is able to provide a linear description of $\X$:
\begin{align}
	& w \leq \Phi(\CS) + \sum_{j \in [n] \backslash \CS} \rho_{j}(\CS)x_{j} - \sum_{j \in \CS} \rho_{j}([n]\backslash j)(1-x_j),~\forall~\CS\subseteq [n] ,\label{sc1}\\
	& w \leq  \Phi(\CS) + \sum_{j \in [n]\backslash\CS}\rho_j(\varnothing) x_j - \sum_{j \in \CS} \rho_{j}(\CS \backslash j)(1-x_j),
	~\forall~\CS\subseteq [n]. \label{sc2}
\end{align}
\begin{proposition}[\cite{Nemhauser1981}]\label{prop:X-SC}
	$\X = \{(w, x) \in \R \times \{0,1\}^n : \eqref{sc1} \text{ or } \eqref{sc2}\}$.
\end{proposition}

In the subsequent discussion, we use notations $\CS$-\SI and $\CS$-\SIbar to denote inequalities \eqref{sc1} and \eqref{sc2}, respectively, defined by $\CS$. 
When the context is clear, we abbreviate them as \SI and  \SIbar.
The result in \cref{prop:X-SC} enables to derive the following \MILP reformulation of problem \eqref{prob:GBD-pre} in which 
$w_i\leq \phi_{i}(x)$, $i \in [m]$, are replaced by the corresponding submodular inequalities
\begin{subequations}\label{prob:SC}
	\begin{align}
		\max \ & \sum_{i\in [m]}b_i w_i -\sum_{j \in [n]} f_jx_j \nonumber\\
		\text{s.t.} \   
		& w_i \le \Phi_i({\CS}) + \sum_{j \in [n]\backslash{\CS}} \rho_{ij}(\CS) x_j - \sum_{j \in \CS} \rho_{ij}([n]\backslash j) (1 - x_j), \ 
		\forall \ i \in [m], \ {\CS} \subseteq [n], 
		\label{cons:SC1-all}\\
		& w_i \le \Phi_i({\CS}) + \sum_{j \in [n]\backslash{\CS}} \rho_{ij}(\varnothing) x_j - \sum_{j \in {\CS}} \rho_{ij}(\CS\backslash j) (1 - x_j), \ \forall \ i \in [m], \ \CS \subseteq [n],
		\label{cons:SC2-all}\\
		& x_j \in \{0,1\}, \ \forall \ j \in [n]. \nonumber
	\end{align}
\end{subequations}
Here, \eqref{cons:SC1-all} and \eqref{cons:SC2-all} are inequalities \eqref{sc1} and \eqref{sc2}, respectively, with subscript $i$. 
In \cref{sect:Imple}, we will present a \BnC approach to solve problem \eqref{prob:SC} in which the {exponential families} of submodular inequalities \eqref{cons:SC1-all} and \eqref{cons:SC2-all} are separated on the fly.

The following proposition shows that the above \MILP formulation can be further simplified by incorporating only a small subset of the inequalities.
The proof can be found in Appendix \ref{sect:appendix-SC-simplification}. 
\begin{proposition}\label{prop:SC-simplification}
	For $\CS = \{j_1, j_2, \ldots, j_s\} \subseteq [n]$ with $j_1 < j_2 <  \cdots < j_s$, let $k_{\CS}$ and $\CS^{\gamma}$ be defined as \eqref{def:Sgamma} and let
	\begin{equation}\label{def:k-S-ell}
		\bar{\CS}^\gamma = \left\{
		\begin{aligned}
			& \CS && \text{ if } s \le \gamma ; \\
			& \CS\cup \{j_{\gamma+1},  j_{\gamma+1}+1, \ldots, n \} && \text{otherwise}.
		\end{aligned}
		\right.
	\end{equation}
	Then the following two statements hold: 
	(i) $\CS$-\SI is equivalent to $\CS^\gamma$-\SI;
	(ii) $\CS$-\SIbar is dominated by  $\bar{\CS}^\gamma$-\SIbar.
\end{proposition}
\noindent Let 
\begin{equation}\label{Qdef}
	\CQ^1 = \{\CS \subseteq [n]: |\CS|\le \gamma\}~\text{and}~\CQ^2 = \CQ^1 \cup \{\bar{\CS}^{\gamma}: \CS \subseteq [n] ~\text{and}~ |\CS| = \gamma + 1\}.
\end{equation}
From \cref{prop:SC-simplification}, we can provide a more compact linear description of $\X$, where the submodular inequalities in \eqref{sc1} and \eqref{sc2} are simplified as
\begin{align}
	& w \leq \Phi(\CS) + \sum_{j \in [n] \backslash \CS} \rho_{j}(\CS)x_{j} - \sum_{j \in \CS} \rho_{j}([n]\backslash j)(1-x_j),~\forall~\CS \in \CQ^1 ,\label{sc1-sim}\\
	& w \leq  \Phi(\CS) + \sum_{j \in [n]\backslash\CS}\rho_j(\varnothing) x_j - \sum_{j \in \CS} \rho_{j}(\CS \backslash j)(1-x_j),~\forall~\CS \in \CQ^2. \label{sc2-sim}
\end{align}
\subsection{A compact \MILP formulation for \eqref{prob:original-max-S} with partially proportional rule}

Next, we shall demonstrate the strength of inequalities \eqref{sc1-sim} and \eqref{sc2-sim} by considering the special case that $\gamma = 1$, which arises in  \eqref{prob:original-max-S} with partially proportional rule \citep{Biesinger2016,Fernandez2017}.
We first notice that, in this case, $\CQ^1 = \{\varnothing \} \cup \{  \{\ell\} \, : \, \ell \in [n]\}$ and 
$\CQ^2 =\CQ^1 \cup \{\{k, \ell, \ell+1, \dots, n\} \, : \,  1 \le k < \ell \le n\}$.
Moreover,
	\begin{itemize}
		\item [\rm{(i)}] The $\varnothing$-\SI, $\varnothing$-\SIbar, and $\{\ell\}$-\SIbar ($\forall~\ell \in[n]$) are all equivalent to
		\begin{equation}\label{sc1-gamma1-0}
			w \le \sum_{j \in [n]} \frac{u_j}{u_j + u_0} x_j.
		\end{equation}	
		\item [\rm{(ii)}] For each $\ell \in \{ 2, 3, \dots, n\}$, the $\{\ell\}$-\SI is equivalent to 
		\begin{equation}\label{sc1-gamma1}
			\begin{aligned}
			w &\le \frac{u_{\ell}}{u_{\ell} + u_0} + \sum_{j = 1}^{\ell-1} \left( \frac{u_j}{u_j + u_0} - \frac{u_{\ell}}{u_{\ell} + u_0}\right) x_j  \\
			&=\frac{u_{\ell}}{u_{\ell} + u_0} + \sum_{j = 1}^{n} \left( \frac{u_j}{u_j + u_0} - \frac{u_{\ell}}{u_{\ell} + u_0}\right)^+ x_j .
			\end{aligned}
		\end{equation}
		The $\{1\}$-\SI and $\{2\}$-\SI are equivalent.
		\item [\rm{(iii)}] For any $k$, $\ell \in [n]$ and $k < \ell$, the
		$\{ k,\ell , \ell+1, \ldots, n\}$-\SIbar reads 
		\begin{equation}
			\begin{aligned}
				w 
				& \le \frac{u_\ell}{u_\ell + u_0} + \sum_{j =1}^{\ell-1} \left( \frac{u_j}{u_j + u_0} - \frac{u_\ell}{u_\ell + u_0}\right) x_j +
				\frac{u_\ell}{u_\ell + u_0} \cdot \sum_{j \in [\ell - 1] \backslash k} x_j, 
			\end{aligned}
		\end{equation}
		and is dominiated by $\{\ell\}$-\SI.
	\end{itemize}

\begin{example}\label{exp:SM}
	Consider the set $\X^1$ defined in  \cref{exp:GBD-weak}.
	Then inequality \eqref{sc1-gamma1-0} reads
	\begin{equation*}
		\begin{aligned}
			w 
			& = \Phi(\varnothing) + \sum_{j = 1,2,3} \rho_j(\varnothing) x_j
			= 0 + \frac{5}{15} x_1 + \frac{4}{14} x_2 + \frac{3}{13} x_3,
		\end{aligned}
	\end{equation*}
	which is exactly the inequality  \eqref{ineq:example-SM-1} and is stronger than the Benders cut \eqref{ineq:example-GBD-1}.
\end{example}

Notice that as we assume $u_{n+1} = 0$, inequality \eqref{sc1-gamma1-0} can be written as $w \le \frac{u_{n+1}}{u_{n+1} + u_0} + \sum_{j = 1}^{n} ( \frac{u_{j}}{u_{j} + u_0} - \frac{u_{n+1}}{u_{n+1} + u_0})^+ x_j$, and thus 
inequalities \eqref{sc1-gamma1-0} and \eqref{sc1-gamma1} can be unified into
\begin{equation}\label{sc1-gamma1-unified}
	w \le \frac{u_{\ell+1}}{u_{\ell+1} + u_0} + \sum_{j = 1}^{n} \left( \frac{u_j}{u_j + u_0} - \frac{u_{\ell+1}}{u_{\ell+1} + u_0}\right)^+ x_j, \ \forall~\ell \in [n].
\end{equation}
Thus, from \cref{prop:SC-simplification}, $\X$ can be described by only $n$ constraints in \eqref{sc1-gamma1-unified}.
Therefore, for  the special case that $\gamma_i = 1$ for all $i \in [n]$, i.e., the customers follow the partially binary rule \citep{Biesinger2016, Fernandez2017}, we can develop a compact \MILP formulation for \eqref{prob:original-max-S}:
\begin{equation}\label{prob:SC-gamma1}
	\begin{aligned}
		\max \ & \sum_{i\in [m]}b_i w_i -\sum_{j \in [n]} f_jx_j \\
		\text{s.t.} \   
		& w_i \le \frac{u_{i(\ell+1)}}{u_{i(\ell+1)} + u_{i0}} + \sum_{j = 1}^n \left( \frac{u_{ij}}{u_{ij} + u_{i0}} - \frac{u_{i(\ell+1)}}{u_{i(\ell+1)} + u_{i0}}\right)^+ x_j, \ \forall \ i \in [m], \ \ell \in [n], \\
		& x_j \in \{0,1\}, \ \forall \ j \in [n]. 
	\end{aligned}
\end{equation}
Here we assume $u_{i(n+1)}=0$ for all $i \in [m]$.

The following theorem demonstrates the strength of formulation \eqref{prob:SC-gamma1} in providing a strong \LP relaxation bound.
The proof of this theorem is presented in Appendix \ref{sect:appendix-convex-hull-gamma1}. 
\begin{theorem}\label{thm:convex-hull-gamma1}
	If $\gamma = 1$, then $\conv(\X)=\left\{(w,x) \in \R \times [0,1]^n: \eqref{sc1-gamma1-unified}\right\}$.
\end{theorem}
\noindent Observe that the Benders cuts in \eqref{ineq:GBD} provide a linear  description of $\X$ but may not provide a linear description of $\conv(\X)$.
Therefore, \cref{thm:convex-hull-gamma1} demonstrates that the \LP relaxation of problem \eqref{prob:SC-gamma1} is stronger than that of the Benders reformulation of \cite{Lin2021} (with the Benders cuts). 
With this advantage, it can be expected that problem \eqref{prob:SC-gamma1} with the submodular inequalities  is much more computationally solvable than that with the Benders cuts.
In \cref{sect:Compu}, we will further conduct experiments to verify this. 

\begin{remark}
	When $\gamma = 1$, function $\Phi(\CS)$ %
	reduces to {$\max_{j \in \CS} \frac{u_j}{u_j+u_0}$}.
	It is worthwhile mentioning that {\cite{Nemhauser1981}} established the submodularity of function $\max_{j \in \CS} \frac{u_j}{u_j+u_0}$, and derived a linear description of $\X$ using the submodular inequalities in \eqref{sc1-gamma1-unified}.
	Here we further show that the submodular inequalities in \eqref{sc1-gamma1-unified} are able to provide a linear description of $\conv(\X)$; see \cref{thm:convex-hull-gamma1}.
\end{remark}

\section{An \MILP formulation based on lifted submodular inequalities}\label{sect:LSI}

In general, the submodular inequalities, however, could  be weak.
We refer to \cite{Ahmed2011,Shi2022} for the discussion on the weakness of the submodular inequalities in the context of solving a class of submodular maximization problems where  the submodular function can be represented as the composition of a concave function and an affine function.
In the context of solving the \MILP formulaiton \eqref{prob:SC} of \eqref{prob:original-max-S}, the weakness of the submodular inequalities \eqref{sc1} and \eqref{sc2} leads to a weak \LP relaxation, making a large search tree by the \BnC algorithm, as will be demonstrated in the experiments.  
In this section, we shall overcome this weakness by proposing a class of lifted submodular inequalities, obtained by using the sequential lifting technique \citep{Wolsey1976a,Richard2011}. 
The proposed lifted submodular inequalities are shown to be stronger than the submodular inequalities, which enables to derive a new \MILP formulation with a tighter \LP relaxation than formulation \eqref{prob:SC}. 
We also discuss the algorithmic details for the computation of the proposed inequalities.

\subsection{Strong valid inequalities by sequential lifting}

This subsection first reviews some background on sequential lifting and then develops the lifted submodular inequalities.
\subsubsection{Background on sequential lifting}\label{subsubsect:backglifting}

Sequential lifting is a useful procedure to generate strong valid (facet-defining) inequalities for a polyhedron from inequalities that are valid for the polyhedron in a low-dimensional space.
It has been widely employed in developing facet-defining inequalities for various polyhedra. 
We refer to \citep{Wolsey1976a,Richard2011} for a detailed discussion of sequential lifting. 
Here, we apply the sequential lifting technique to derive strong valid inequalities for $\conv(\X)$, where $\X$ is defined in \eqref{def:X}. 

Let $\CS_0$ and $\CS_1 \subseteq [n]$ be  two disjoint subsets, and consider a low-dimensional set
\begin{equation}
	\X(\CS_0, \CS_1) = \left\{ (w,x) \in \X: x_j = 0,~ \forall  ~j ~\in \CS_0, \ x_j = 1, ~\forall~ j~ \in \CS_1\right\},
\end{equation}
obtained by fixing variables with indices in $\CS_0$ and $\CS_1$ to $0$ and $1$ in $\X$, respectively.
Observe that $(-1,\sum_{j \in \CS_1} \boldsymbol{e}_j)$, $(0,\sum_{j \in \CS_1} \boldsymbol{e}_j)$,
$(0, \sum_{j \in \CS_1 \cup \ell} \boldsymbol{e}_j)$, $\ell \in [n] \backslash (\CS_0 \cup \CS_1)$, 
are affinely independent points in $\X(\CS_0, \CS_1)$.
Thus,
\begin{remark}\label{rmk:dim-X-S0-S1}
	The dimension of $\conv(\X(\CS_0, \CS_1))$ is $n + 1- |\CS_0| -|\CS_1|$.
\end{remark}

	Let
	\begin{equation}\label{ineq:seed}
		w \le \alpha_0 + \sum_{j \in [n] \backslash (\CS_0 \cup \CS_1)} \alpha_j x_j
	\end{equation}
	be a valid inequality of $\X(\CS_0, \CS_1)$, also referred to as a \emph{seed inequality}.
	Sequential lifting involves introducing the projected variables (i.e., $x_j$, $j \in \CS_0 \cup \CS_1$) into inequality \eqref{ineq:seed} one at a time, following a specific lifting ordering.
	Specifically, 
	let $j_1, j_2, \dots, j_s$ be an ordering of  $\CS_0 \cup \CS_1$ (where $s=|\CS_0| + |\CS_1|$) 
	and consider 
	\begin{equation}\label{ineq:ell-minus-1}
		\begin{aligned}
			w 
			& \le \alpha_0 + \sum_{j \in [n]\backslash(\CS_0 \cup \CS_1)} \alpha_j x_j 
			+ \sum_{j \in \CS'_0} \alpha_j x_j
			- \sum_{j \in \CS'_1} \alpha_j (1-x_j) 
		\end{aligned}
	\end{equation}
	 to be a so-far generated inequality for $\X(\CS_0\backslash \CS'_0, \CS_1\backslash \CS'_1)$, where $\CS'_0\subseteq \CS_0$, $\CS'_1 \subseteq \CS_1$, $\CS'_0 \cup \CS'_1 = \{j_1, j_2, \ldots, j_{\ell-1} \}$, and  $\ell \in [s]$.
	To lift variable $x_{j_{\ell}}$ into inequality \eqref{ineq:ell-minus-1}, we consider the following two cases.
	\begin{itemize}
		\item {Up-lifting}. If $j_\ell \in \CS_0$, we attempt to find the lifting coefficient $\alpha_{j_{\ell}}$ such that inequality 
		\begin{equation}\label{ineq:ell-S0}
			w \le \alpha_0 + \sum_{j \in [n]\backslash(\CS_0 \cup \CS_1)} \alpha_j x_j 
			+ \sum_{j \in \CS'_0} \alpha_j x_j
			- \sum_{j \in \CS'_1} \alpha_j (1-x_j ) 
			+ \alpha_{j_{\ell}} x_{j_{\ell}}
		\end{equation}
		is valid for $\X(\CS_0 \backslash (\CS'_0 \cup j_{\ell}), \CS_1 \backslash \CS'_1)$. 
		The {smallest coefficient}  $\alpha_{j_{\ell}}$ can be computed by $\alpha_{j_\ell}=-\alpha_0 + \sum_{j \in \CS'_1} \alpha_j + \nu^*$ where 
		\begin{equation}\label{liftprob:up-lifting}
			\nu^* = \max \left\{
			w - \sum_{j \in [n]\backslash(\CS_0 \cup \CS_1)} \alpha_j x_j 
			- \sum_{j \in \CS'_0\cup \CS'_1} \alpha_j x_j
			 : 
			(w,x) \in \X(\CS_0\backslash (\CS'_0\cup j_\ell), \CS_1\cup j_\ell\backslash \CS'_1)
			\right\}.
		\end{equation}
		\item {Down-lifting}. If $j_\ell \in \CS_1$, we attempt to find the lifting coefficient $\alpha_{j_{\ell}}$ such that inequality 
		\begin{equation}\label{ineq:ell-S1}
			w \le \alpha_0 + \sum_{j \in [n]\backslash(\CS_0 \cup \CS_1)} \alpha_j x_j 
			+ \sum_{j \in \CS'_0} \alpha_j x_j
			- \sum_{j \in \CS'_1} \alpha_j (1-x_j) 
			- \alpha_{j_{\ell}} (1-x_{j_{\ell}} )
		\end{equation}
		is valid for $\X(\CS_0 \backslash \CS'_0, \CS_1 \backslash (\CS'_1 \cup j_{\ell}))$. 
		The largest coefficient  $\alpha_{j_{\ell}}$ can be computed by $\alpha_{j_\ell}=\alpha_0 - \sum_{j \in \CS'_1} \alpha_j -\nu^*$ where 
		\begin{equation}\label{liftprob:down-lifting}
			\nu^* = \max \left\{
			w - \sum_{j \in [n]\backslash(\CS_0 \cup \CS_1)} \alpha_j x_j 
			- \sum_{j \in \CS'_0\cup \CS'_1} \alpha_j x_j
			: 
			(w,x) \in \X(\CS_0 \cup j_\ell \backslash \CS'_0, \CS_1\backslash (\CS'_1 \cup j_{\ell}))
			\right\}.
		\end{equation}
	\end{itemize}

Repeating the above lifting procedure until $\ell = s $, we obtain the lifted inequality 
\begin{equation}\label{ineq:lifted}
	w 
	\le \alpha_0 + \sum_{j \in [n]\backslash(\CS_0 \cup \CS_1)} \alpha_j x_j + \sum_{j \in \CS_0} \alpha_j x_j
	- \sum_{j \in \CS_1} \alpha_j (1-x_j),
\end{equation}
which is valid for $\conv(\X)$.  
The following theorem follows from \cref{rmk:dim-X-S0-S1} and the classic result in \cite{Wolsey1976a}.
\begin{theorem}[\cite{Wolsey1976a}]
	\label{prop:lifting-general-facet}
	If inequality \eqref{ineq:seed} defines a $d-$dimensional face of $\conv(\X(\CS_0, \CS_1))$,
	then inequality \eqref{ineq:lifted} defines a face of $\conv(\X)$ of dimension at least $d+|\CS_0| +|\CS_1|$.
	Moreover, if inequality \eqref{ineq:seed} defines a facet of $\conv(\X(\CS_0, \CS_1))$,
	then inequality \eqref{ineq:lifted} defines a facet of $\conv(\X)$ as well.
\end{theorem}
It should be noted that inequality \eqref{ineq:lifted} depends on the lifting ordering $j_1, j_2, \dots, j_s$ of $\CS_0 \cup \CS_1$, that is,  different lifting orderings may lead to different lifted inequalities.

\subsubsection{Two families of lifted submodular inequalities}
Next, we derive two families of lifted submodular inequalities for $\conv(\X)$. 
Given $\CS \subseteq [n]$, the inequality  
\begin{equation}\label{tmp_low_dim1}
	w \le \Phi(\CS) + \sum_{j \in [n] \backslash \CS} \rho_j(\CS) x_j
\end{equation} 
is a valid inequality for $\X(\varnothing,\CS)$ {(as it is a submodular inequality for $\X(\varnothing,\CS)$)}.
Let $j_1, j_2, \dots, j_s$ be an ordering of $\CS$. 
Then, using the {down-lifting} procedure in \cref{subsubsect:backglifting}, we obtain the valid inequality 
\begin{equation}\label{sc1-lifted}
	w \le \Phi(\CS) + \sum_{j \in [n] \backslash \CS} \rho_j(\CS) x_j - \sum_{j \in \CS} \eta_j(1 - x_j)
\end{equation}
for $\X$, where for each $\ell\in [s]$, the lifting coefficient is computed as follows
\begin{equation}\label{sc1liftprob}
	\eta_{j_{\ell}} = \Phi(\CS) - \sum_{\tau =1}^{\ell-1} \eta_{j_\tau}  - \max \left\{
	w
	- \sum_{j \in [n] \backslash \CS} \rho_j(\CS) x_j 
	- \sum_{\tau =1}^{\ell-1} \eta_{j_\tau} x_{j_\tau}  : 
	(w,x) \in \X(\{j_\ell\}, \{j_{\ell + 1}, j_{\ell + 2}, \dots, j_s\})
	\right\}.
\end{equation}

Similarly, given $\CS \subseteq [n]$, the inequality  
\begin{equation}\label{tmp_low_dim2}
	w \le \Phi(\CS) - \sum_{j \in \CS} \rho_j(\CS \backslash j) (1 - x_j)
\end{equation}
is a valid inequality for $\X([n]\backslash\CS,\varnothing)$ {(as it is a submodular inequality for $\X([n]\backslash\CS, \varnothing)$)}.
Let $j_1, j_2, \dots, j_{n-s}$ be an ordering of $[n] \backslash\CS$. 
Then, using the {up-lifting} procedure in \cref{subsubsect:backglifting}, we obtain the valid inequality 
\begin{equation}\label{sc2-lifted}
	w \le \Phi(\CS) - \sum_{j \in \CS} \rho_j(\CS \backslash j) (1 - x_j) + \sum_{j \in [n]\backslash \CS} \zeta_j x_j
\end{equation}
for $\X$, where for each $\ell \in [n-s]$, the lifting coefficient is computed as follows 
\begin{equation}
	\label{sc2liftprob}
	\begin{aligned}
	& \zeta_{j_{\ell}} = - \Phi(\CS) + \sum_{j \in \CS} \rho_j(\CS \backslash j) +\\
	&  \qquad\qquad\quad \max \left\{
	w - \sum_{j \in \CS} \rho_j(\CS \backslash j) x_j - \sum_{\tau = 1}^{\ell - 1} \zeta_{j_{\tau}} x_{j_{\tau}} : 
	(w,x) \in \X(\{j_{\ell + 1}, j_{\ell + 2}, \dots, j_{n-s}\}, \{j_{\ell}\})
	\right\}.
	\end{aligned}
\end{equation}

For any $\CS$, it is simple to check that the $n-|\CS| + 1$ affinely independent points $(\Phi(\CS), \sum_{j \in \CS}\boldsymbol{e}_j)$, $(\Phi(\CS\cup t), \sum_{j \in \CS\cup t}\boldsymbol{e}_j)$, $\forall~t \in [n]\backslash\CS$, are in $\conv(\X(\varnothing,\CS))$ fulfilling \eqref{tmp_low_dim1} with equality; 
and  the $|\CS| + 1$ affinely independent points $(\Phi(\CS), \sum_{j \in \CS}\boldsymbol{e}_j)$, $(\Phi(\CS\backslash t), \sum_{j \in \CS\backslash t}\boldsymbol{e}_j)$, $\forall~t \in \CS$, are in $\conv(\X([n]\backslash \CS, \varnothing))$ fulfilling \eqref{tmp_low_dim2} with equality.
Therefore, inequalities \eqref{tmp_low_dim1} and \eqref{tmp_low_dim2} are facet-defining for $\conv(\X(\varnothing, \CS))$ and $\conv(\X([n]\backslash \CS, \varnothing))$, respectively. 
From \cref{prop:lifting-general-facet}, we immediately obtain the following result.

\begin{proposition}\label{cor:sc-lifted-facet-def}
	Inequalities \eqref{sc1-lifted} and \eqref{sc2-lifted} are facet-defining for $\conv(\X)$.
\end{proposition}

\subsection{Strength of lifted submodular inequalities}

We first characterize the bounds of coefficients in non-trivial facet-defining inequalities of the convex hull of the mixed 0-1 set defined by a generic \emph{submodular function}.
\begin{theorem}\label{thm:nontrivial}
	Let $\Psi: 2^{[n]} \rightarrow \R$ be a real-valued submodular function, and $\bar{\rho}_j(\CS)=\Psi(\CS \cup j) - \Psi(\CS)$ for $\CS \subseteq [n]$ and $j \in [n] \backslash \CS$.
	Define 
	\begin{equation*}
		\Y = \left\{(w,x) \in \R \times \{0,1\}^{n}: w \le \psi(x)\right\},
	\end{equation*}
	where $\psi(x) = \Psi(\CS^{x})$ and $\CS^{x} = \{j \in [n]: x_j = 1\}$.
	Then every nontrivial facet-defining inequality of $\conv(\Y)$ can be represented as $w \leq c+ \alpha^\top x $,
	where $c \geq 0$ and  $\bar{\rho}_j([n]\backslash j) \leq \alpha_j \leq \bar{\rho}_j(\varnothing) $ for all $j \in [n]$.
\end{theorem} 
\begin{proof}
	Let $\beta w \leq c+\alpha^\top x$  be a nontrivial facet-defining inequality of $\conv(\Y)$.
	{Since the projection of $\Y$ onto the $x$ space reads $\{0,1\}^n$, $\beta\neq 0$ must hold. }
	For any $w \in \mathbb{R}_-$, it follows $(w, \boldsymbol{0}) \in \Y$, which implies $\beta > 0$.
	By scaling the inequality if needed, we can assume $\beta =1$.
	Next, we further show that $\bar{\rho}_j([n]\backslash j) \leq \alpha_j \leq \bar{\rho}_j(\varnothing) $ holds for all $j \in [n]$. 
	Observe that for a point $(w, x) \in \Y$ satisfying  $ w =c+ \alpha^\top x$, $w = \psi(x)$ must hold. 
	Since $ w \leq c+\alpha^\top x$ is a nontrivial facet-defining inequality of $\conv(\Y)$ and differs from $x_j \leq 1$,
	there exists a point $(\psi(\bar{x}), \bar{x})$ in $\Y$ satisfying $\bar{x}_j = 0$ and $\psi(\bar{x}) =c+ \alpha^{\top} \bar{x}$.
	Clearly, point $(\psi(\bar{x}+\boldsymbol{e}_j),\bar{x} + \boldsymbol{e}_j)$ is also in $\Y$.  
	Substituting $(\psi(\bar{x}+\boldsymbol{e}_j),\bar{x} + \boldsymbol{e}_j)$ into $ w \leq c+\alpha^\top x $ and subtracting $\psi(\bar{x}) = c+\alpha^\top \bar{x}$, we obtain $\alpha_j \geq \psi(\bar{x}+\boldsymbol{e}_j)-\psi(\bar{x}) \geq \bar{\rho}_j([n]\backslash j) $, where the last inequality follows from {the submodularity of $\Psi$}.
	On the other hand, since $ w \leq c+ \alpha^\top x $ is a nontrivial facet-defining inequality of $\conv(\Y)$ and differs from $x_j \geq 0$,
	there exists a point $(\psi(\hat{x}), \hat{x})$ in $\Y$ satisfying $\hat{x}_j = 1$ and $ \psi(\hat{x}) = c+ \alpha^\top \hat{x}$.
	Clearly, point $(\psi(\hat{x}-\boldsymbol{e}_j), \hat{x}-\boldsymbol{e}_j)$ is also in $\Y$.
	Substituting $(\psi(\hat{x}-\boldsymbol{e}_j), \hat{x}-\boldsymbol{e}_j)$ into $ w \leq c+ \alpha^\top x $ and subtracting $ \psi(\hat{x}) = c+ \alpha^\top \hat{x} $, we obtain $\alpha_j \leq \psi(\hat{x})-\psi(\hat{x}-\boldsymbol{e}_j) \leq \bar{\rho}_j(\varnothing)$, where the last inequality follows from {the submodularity of $\Psi$}.
\end{proof}
\cref{thm:nontrivial} enables to provide lower and upper bounds for the lifting coefficients in inequalities \eqref{sc1-lifted} and \eqref{sc2-lifted}.
In particular, 
since $\Phi$ is a submodular function and inequalities  \eqref{sc1-lifted} and \eqref{sc2-lifted}  are nontrivial facet-defining inequalities of $\conv(\Y)$, we can derive from \cref{thm:nontrivial} that $\eta_j \geq \rho_j([n]\backslash j)$ for all $j \in \CS$ and $\zeta_j \leq \rho_j(\varnothing)$ for all $j \in [n]\backslash\CS$.
As a result, 
\begin{corollary}\label{cor:nontrivial-X}
	For any $\CS \subseteq [n]$, inequalities \eqref{sc1-lifted} and \eqref{sc2-lifted} are stronger than the submodular inequalities \eqref{sc1} and \eqref{sc2}, respectively.
\end{corollary}
As an immediate result of \cref{cor:nontrivial-X},
the formulation \eqref{prob:SC} for \eqref{prob:original-max-S} can be  strengthened as
\begin{subequations}\label{prob:SC-lifted}
	\begin{align}
			\max \ & \sum_{i\in [m]}b_i w_i -\sum_{j \in [n]} f_jx_j \nonumber\\
			\text{s.t.} \   
			& w_i \le \Phi_i(\CS) + \sum_{j \in [n]\backslash\CS} \rho_{ij}(\CS) x_j - \sum_{j \in \CS} \eta^\sigma_{ij} (1 - x_j), \
			\forall \ i \in [m], \ \CS \subseteq [n], \ \sigma \in \order(\CS), \label{cons:SC1-lifted-all}\\
			& w_i \le \Phi_i(\CS) -\sum_{j \in\CS} \rho_{ij}(\CS \backslash j) (1 - x_j) + \sum_{j \in [n]\backslash\CS} \zeta^\sigma_{ij} x_j, \ \forall \ i \in [m], \ 
			\CS \subseteq [n], \ \sigma \in \order([n]\backslash \CS),
			\label{cons:SC2-lifted-all}\\
			& x_j \in \{0,1\}, \ \forall \ j \in [n]. \nonumber
		\end{align}
\end{subequations}
Here, for $\CS \subseteq [n]$, $\order(\CS)$ represents the set of all orderings of $\CS$; 
for each inequality of \eqref{cons:SC1-lifted-all} (respectively, \eqref{cons:SC2-lifted-all}) with the lifting ordering $\sigma$ of $\CS$ (respectively, $[n]\backslash \CS$),
the lifting coefficients $\eta^\sigma_{ij}$ (respectively, $\zeta^\sigma_{ij}$) are computed by solving the lifting problems of the form \eqref{sc1liftprob} (respectively, \eqref{sc2liftprob}).
By \cref{cor:nontrivial-X},  formulation \eqref{prob:SC-lifted} is much stronger than formulation \eqref{prob:SC} in terms of providing a better \LP relaxation bound, which makes it much more computationally solvable in a \BnC framework; see \cref{subsect:compu-submod-lifted} further ahead.

\subsection{An algorithm for the computation of lifted (submodular)  inequalities}
\label{subsect:complexity-calculation-lifted}

To compute the lifting coefficients of the lifted inequality \eqref{ineq:lifted} (or \eqref{sc1-lifted} and \eqref{sc2-lifted} for the special cases), we need to solve a sequence of lifting problems \eqref{liftprob:up-lifting} or \eqref{liftprob:down-lifting}, which take the form of
\begin{equation}\label{def:v-N0-N1}
	\max\left\{ 
	w - \sum_{j \in [n] \backslash (\CN_0 \cup \CN_1)} \alpha_j x_j : w\leq \phi(x), ~x_j =0,~\forall~j \in \CN_0, ~x_{j}=1,~\forall~j \in \CN_1
	\right\},
\end{equation}
where $\CN_0$ and $\CN_1$ are two disjoint subsets of $[n]$.
Problem \eqref{def:v-N0-N1} can be simplified as follows.
First, we observe that $w=\phi(x)$ holds for any optimal solution $(w,x)$ of problem \eqref{def:v-N0-N1}, which, together with the definition of $\phi(x) $ in \eqref{phidef}, enables to rewrite problem \eqref{def:v-N0-N1} as 
\begin{equation}\label{v-N0-sN1-xy}
	\begin{aligned}
		\max &  \left\{ \frac{\sum_{j \in [n]} u_j y_j}{\sum_{j \in [n]} u_j y_j + u_0} - \sum_{j \in [n] \backslash (\CN_0 \cup \CN_1)} \alpha_j x_j : 
		\sum_{j \in [n]} y_j \le \gamma, \ 
		y_j \le x_j, \ \forall \ j \in [n], \ \right.\\
		& \left. 
		\qquad \qquad   
		\phantom{\sum_{j \in [n]} y_j \le \gamma}
		x \in \{0,1\}^n, \ 
		y \in \{0,1\}^n, \
		x_j = 0,\  \forall \ j \in \CN_0, \ 
		x_j = 1, \ \forall \ j \in \CN_1
		\right\}.
	\end{aligned}
\end{equation}
Second, for each $j \in [n] \backslash (\CN_0 \cup \CN_1)$, we can set $x_j = 1$ if $\alpha_j \leq 0$ and $x_j = y_j$ otherwise in problem \eqref{v-N0-sN1-xy}.
Thus, by removing the fixed variables and redundant constraints from problem \eqref{v-N0-sN1-xy}, we can obtain an equivalent problem taking the form of    
\begin{equation}\label{prob:lift-NPhard}
	\nu^* = \max \left\{ 
	\frac{\sum_{j = 1}^{p} u_j y_j}{\sum_{j = 1}^p u_j y_j + u_0} - \sum_{j = 1}^{q} \alpha_j y_j : 
	\sum_{j = 1}^p y_j \le \gamma, \ 
	y \in \{0,1\}^p
	\right\},
\end{equation}
where $q \leq p$.
Without loss of generality, we assume $u_j \in \mathbb{Z}_+$ for all $j \in {[p] \cup 0}$ throughout this subsection.
Notice that here we do not require $u_1 \geq u_2 \ge \cdots \geq u_p$.

The following theorem shows the intrinsic difficulty of solving problem \eqref{prob:lift-NPhard}.

\begin{theorem}\label{thm:np-hard}
	Problem \eqref{prob:lift-NPhard} is NP-hard even when $\gamma \ge p$.
\end{theorem}
\noindent The proof of \cref{thm:np-hard} is provided in Appendix \ref{sect:appendix-nphard}.
\cref{thm:np-hard} shows that unless $\text{P}= \text{NP}$, there does not exist a polynomial time algorithm to solve problem \eqref{prob:lift-NPhard}. 
In the following, we will therefore concentrate on the development of an efficient pseudo-polynomial time algorithm to solve problem \eqref{prob:lift-NPhard}, and thus to compute the lifted inequality \eqref{ineq:lifted}.

For $\tau =0, 1, \ldots, \gamma$, define %
\begin{align}
	& f(\tau) = \max \left\{ \sum_{j = q+1}^{p} u_j y_j: 
	\sum_{j = q+1}^p y_j \le \tau, \ y_j \in \{0,1\}, \ \forall \ j = q+1, q+2, \dots, p
	\right\}, \label{f_taudef}
\end{align}
and for $t=0, 1, \ldots, q$, {$\lambda \in\mathbb{Z}_+$}, and $\tau= 0, 1, \dots, \min\{\gamma, t\}$, define
\begin{align}
	& Z_t(\lambda,\tau) = \max \left\{
	\frac{\sum_{j = 1}^t u_j y_j + \lambda}{\sum_{j = 1}^t u_j y_j + \lambda + u_0} - \sum_{j = 1}^t \alpha_j y_j: 
	\sum_{j = 1}^t y_j = \tau, \ y \in \{0,1\}^t
	\right\}.\label{Z_tdef}
\end{align}
Then problem \eqref{prob:lift-NPhard} reduces to
\begin{equation}\label{prob:lift-NPhard-tau}
	\nu^*= \max \left\{
	Z_q(f(\gamma - \tau), \tau): \tau = 0, 1, \dots, \min\{\gamma, q\}
	\right\}.
\end{equation}
Thus, problem \eqref{prob:lift-NPhard} can be solved by computing $\{f(\tau)\}$ and $\{Z_{t}(\lambda, \tau)\}$.

Letting $j_{q+1}, j_{q+2}, \ldots, j_{p}$ be a permutation of $\{q+1, q+2, \ldots, p\}$ such that $u_{j_{q+1}} \geq u_{j_{q+2}} \geq \cdots \geq u_{j_p}$, a closed formula for $f(\tau)$ can be given by 
\begin{equation}\label{fdef2}
	f(\tau)= \sum_{r=q+1}^{\min\{q+\tau, p\}} u_{j_{r}}.
\end{equation}
As for $Z_{t}(\lambda, \tau)$, we first give two trivial cases for which a closed formula can be derived:
\begin{equation}\label{Zinit}
	Z_{t}(\lambda, \tau)=\left\{
	\begin{aligned}
		& \frac{\lambda}{\lambda + u_0}, &~\text{if}~\tau =0;\\
		& \frac{\sum_{j = 1}^{\tau} u_j + \lambda}{\sum_{j = 1}^{\tau} u_j + \lambda + u_0} - \sum_{j = 1}^{\tau} \alpha_j, & ~\text{if}~\tau=t.
	\end{aligned}
	\right. 
\end{equation} 
For the other cases, we have the following recursive formula
\begin{equation}
	 Z_t(\lambda, \tau) = \max \left\{ Z_{t-1}(\lambda, \tau), \ - \alpha_t + Z_{t-1}(\lambda + u_t, \tau - 1)\right\},\label{Zrecur}
\end{equation}
where $t =  1, 2, \dots, q$ and $\tau = 1, 2, \dots, \min\{\gamma, t - 1\}$.

We now analyze the computational complexity of solving problem \eqref{prob:lift-NPhard}. 
Given any $t \in \{1, 2, \dots, q\}$ and $\tau \in \{1, 2, \dots, \min\{\gamma, t-1\}\}$, let $Z_{t'}(\lambda', \tau')$ be encountered during the backtracking of \eqref{Zrecur} where $t' \leq  t-1$, $\tau' \leq \tau$, and $\lambda' = \lambda + u_{j_1}+ u_{j_2}+\cdots + u_{j_{\tau-\tau'}}$ holds for some distinct $j_1, j_2, \ldots, {j_{\tau-\tau'}} \in \{t'+1, t'+2,\ldots, t\}$.
Obviously, $\lambda' \leq \lambda + M_\tau$ {holds}, where $M_\tau$ is the summation of the $\tau$ largest elements in $u_1, u_2, \ldots, u_t$.
Therefore, the computation of $Z_t(\lambda, \tau)$ can be implemented with the complexity of $\mathcal{O}(t\tau({\lambda+}M_\tau))$.
It is easy to see that using the sorting $u_{j_{q+1}} \geq u_{j_{q+2}} \geq\cdots \geq u_{j_p}$, the computation of $f(\tau)$, $\tau =0, 1, \ldots, \min\{p-q, \gamma\}$ in \eqref{fdef2} can be done in $\mathcal{O}(\min\{p-q, \gamma\})$.
Therefore, the computation of $\nu^*$ in \eqref{prob:lift-NPhard-tau} can be implemented with the complexity of $\mathcal{O}((p-q)\log (p-q)+ q\gamma M^*)$,
where $M^*$ is the summation of the $\gamma$ largest elements in $u_1, u_2, \ldots, u_p$.
\begin{proposition}
	Problem \eqref{prob:lift-NPhard} can be solved in $\mathcal{O}((p-q)\log (p-q)+ q\gamma M^*)$.
\end{proposition}

To compute the lifted inequality \eqref{ineq:lifted}, one needs to solve $|\CS_0|+|\CS_1|$ lifting problems of the form \eqref{prob:lift-NPhard}.
It should be noted that previously computed $\{Z_{t}(\lambda, \tau)\}$ can be used for the computation of the current lifting coefficient. 
Therefore, the computation of the lifted inequality \eqref{ineq:lifted} can be implemented with the complexity of $\mathcal{O}(n^2\log n+ n\gamma \bar{M}^*)$,  where $\bar{M}^*$ is the summation of the $\gamma$ largest elements in $u_1, u_2, \ldots, u_n$.

\section{Implementation details of the \BnC algorithms}\label{sect:Imple}

In this section, we develop the \BnC algorithms for solving the \MILP formulations \eqref{prob:SC} and \eqref{prob:SC-lifted} of \eqref{prob:original-max-S}.

\subsection{A two-stage approach for the implementation}
We use  a two-stage approach to implement the \BnC algorithms. 
In the first stage, we solve the \LP relaxation of formulation \eqref{prob:SC} (respectively,  \eqref{prob:SC-lifted}) using a cutting plane algorithm.
Specifically, at each iteration, we solve the \LP problem and add inequalities \eqref{cons:SC1-all}-\eqref{cons:SC2-all} (respectively,  \eqref{cons:SC1-lifted-all}-\eqref{cons:SC2-lifted-all}) violated by the current solution to the \LP problem (where the separation procedure for inequalities \eqref{cons:SC1-all}-\eqref{cons:SC2-all} and \eqref{cons:SC1-lifted-all}-\eqref{cons:SC2-lifted-all} will be detailed in the next {subsection}).
The procedure is repeated until no more violated cut is found.
During the first stage, we also implement a rounding  algorithm to find a  high-quality integral  feasible solution for problem \eqref{prob:SC} or \eqref{prob:SC-lifted}.
In particular, at each iteration, we round the values $x_j^*$ of the current solution $x^*$ to their nearest integers, which guarantees to obtain a feasible solution. 

At the end of the first stage, all constructed inequalities saturated by the final optimal solution of the \LP problem are added as constraints to the relaxed \MILP problem of  \eqref{prob:SC} or \eqref{prob:SC-lifted}.
In addition, the best-found integral feasible solution obtained in the first stage serves as an initial feasible solution for the problem. 
We then go to the second stage and use a \BnC algorithm to solve problem \eqref{prob:SC}  (respectively, problem \eqref{prob:SC-lifted}) in which inequalities  \eqref{cons:SC1-all}-\eqref{cons:SC2-all} (respectively, inequalities \eqref{cons:SC1-lifted-all}-\eqref{cons:SC2-lifted-all}) are separated for integral solutions encountered inside the search tree.

The motivation for implementing the first stage is twofold.
First, it can collect effective cuts for the initialization of the \MILP problem \eqref{prob:SC} or \eqref{prob:SC-lifted}, which helps to favor root-node variable fixing
and the construction of internal cuts in the subsequent \BnC stage; see \cite{Fischetti2016,Bodur2017a} for detailed discussions.
Second, the (possibly) high-quality feasible solution found in the first stage can provide a tight lower bound {that helps to prune} the uninteresting part of the search tree constructed by the \BnC algorithm, thereby improving the overall solution process.

\subsection{Separation of submodular and lifted submodular inequalities}
\label{subsect:sepa}

Given a point $(w^\ast, x^\ast) \in \R \times [0,1]^n$, the separation problem of the submodular inequalities \eqref{cons:SC1-all}-\eqref{cons:SC2-all} or lifted submodular inequalities \eqref{cons:SC1-lifted-all}-\eqref{cons:SC2-lifted-all} asks to find inequalities violated by $(w^\ast, x^\ast) \in \R \times [0,1]^n$ or prove none exists.
For simplicity of discussions, here we omit the index $i$ and consider the separation of the submodular inequalities \eqref{sc1}-\eqref{sc2} and the lifted submodular inequalities \eqref{sc1-lifted} and \eqref{sc2-lifted}.

We first discuss the separation of the submodular inequalities   \eqref{sc1}-\eqref{sc2}.
Note that for $x^\ast \in \{0,1\}^n$, an exact separation of inequalities \eqref{sc1}-\eqref{sc2} can be done by computing ${\Phi(\CS)}$ where $\CS := \{ j \in [n]\, : \, x_j^* =1\}$, and checking whether $w^\ast > {\Phi(\CS)}$ holds.
If $w^\ast >  {\Phi}(\CS)$, the submodular inequalities  \eqref{sc1}-\eqref{sc2} defined by $\CS$ are violated by $(w^\ast, x^\ast) $; otherwise, no violated inequality exists.

For $x^\ast \notin \{0,1\}^n$ that appears in the stage of solving the \LP relaxation of problem \eqref{prob:SC}, it is generally difficult to solve the separation problem exactly.
We therefore use a heuristic algorithm detailed as follows. 
{Let $\CS_1:=\{ j \in [n]\,:\, x_j^* =1 \}$, $\CF := \{j \in [n]: 0 < x^\ast_j < 1\}$, and $j_1, j_2, \dots, j_{|\CF|}$ be a permutation of $\CF$ such that $x^\ast_{j_1} \geq x^\ast_{j_2}\ge \cdots \ge x^\ast_{j_{|\CF|}}$. Initially, we set $\CS:=\CS_1$.}
Then, for each $\ell = 1, 2, \dots, |\CF|$, we iteratively check whether the submodular inequality  \eqref{sc1} or  \eqref{sc2} defined by $\CS \cup j_{\ell}$ can achieve a larger violation than {that achieved by $\CS$}, and if so, we update $\CS:=\CS\cup j_\ell$.
In the end, if the submodular inequality  \eqref{sc1} or  \eqref{sc2} defined by $\CS$ is violated by $(w^*, x^*)$, we add the inequality into the problem; otherwise, we claim that no violated inequality is found.
Note that for  the submodular  inequality \eqref{sc2} defined by $\CS$, we can apply the argument in \cref{prop:SC-simplification} (ii) to obtain a (possibly)  stronger inequality, i.e., the  $\bar{\CS}^\gamma$-\SIbar, where  $\bar{\CS}^\gamma$ is defined by \eqref{def:k-S-ell}.

Compared with the separation of the submodular inequalities \eqref{sc1}-\eqref{sc2}, the separation of inequalities  \eqref{sc1-lifted} and \eqref{sc2-lifted} is more time-consuming as it requires solving the lifting problems of the form \eqref{def:v-N0-N1} to compute the coefficients of the inequalities.
Observe that  by \cref{thm:nontrivial}, if the submodular inequality \eqref{sc1} (respectively, \eqref{sc2}) defined by $\CS$ is violated by $(w^\ast, x^\ast) $, the lifted submodular inequality \eqref{sc1-lifted} (respectively, \eqref{sc2-lifted}) defined by $\CS$ with any lifting ordering must also be violated by $(w^\ast, x^\ast)$.
Using this observation, our designed separation procedure for the lifted submodular inequalities \eqref{sc1-lifted} and \eqref{sc2-lifted} first applies the separation procedure for the submodular inequalities  \eqref{sc1}-\eqref{sc2} to filter inequalities violated by $(w^*, x^*)$, and then computes the corresponding lifted submodular inequalities.
In particular,
if a violated submodular inequality \eqref{sc1} defined by $\CS$ is found, we compute the corresponding lifted submodular inequality \eqref{sc1-lifted}, where the variables in $\CS$ are lifted according {to the nondecreasing ordering of $x_j^*$}. 
This specific lifting ordering enables to obtain an inequality with a large violation as the earlier a variable is lifted, the better its lifting coefficient will be.
The computation of inequalities \eqref{sc2-lifted}  is conducted in a similar manner where the variables in $[n]\backslash\CS$ are lifted according to the nonincreasing ordering of $x_j^*$, as to obtain an inequality with a large violation.
In order to further save the computational efforts, we apply \cref{prop:SC-simplification} to obtain submodular inequalities \eqref{sc1} and  \eqref{sc2} {with small $\CS$ and $[n]\backslash{\CS}$} (i.e., the $\CS^\gamma$-\SI and $\bar{\CS}^\gamma$-\SIbar) and as such, the number of lifted variables will become smaller.

Note that for the case $\gamma = 1$, the (lifted) submodular inequalities in \eqref{sc1-gamma1-unified} are enough to describe $\conv(\X)$, as demonstrated in \cref{thm:convex-hull-gamma1}. 
In this case, an exact polynomial time separation algorithm exists as we only need to examine whether the $n$ inequalities in  \eqref{sc1-gamma1-unified}  are violated or not. 
Using the result in Lemma B.1 in Appendix \ref{sect:appendix-convex-hull-gamma1},
the separation complexity can be further improved to  $\mathcal{O}(n)$.

\section{Computational Study}
\label{sect:Compu}

In this section, we present the computational results to demonstrate the effectiveness of the proposed \BnC algorithms for solving \eqref{prob:original-max-S}.
To accomplish this, we first perform numerical experiments to compare the performance of the proposed \BnC algorithms based on formulations \eqref{prob:SC} and \eqref{prob:SC-lifted}, where the submodular and lifted submodular inequalities are adopted, respectively.
Then, we present the computational results to demonstrate  the computational efficiency of the proposed \BnC algorithm based on formulation \eqref{prob:SC-lifted} 
over the state-of-the-art \GBD approach in \cite{Lin2021}.

The proposed \BnC algorithms were implemented in Julia 1.9.2 using \Gurobi 11.0.1 as the solver.
\Gurobi was set to run the code in a single-threaded mode, with a time limit of 7200 seconds and a relative \MIP gap tolerance of 0\%.
Parameter  \texttt{LazyConstraints} was set to 1, which enables to add cuts during the \BnC procedure.
Following \cite{Lin2021}, we set parameters \texttt{IntFeasTol}, \texttt{FeasibilityTol}, and \texttt{OptimalityTol} to $10^{-9}$ (as to avoid numerical issues), and  \texttt{Cuts} to 3 (as to trigger the generation of more internal cuts, thereby facilitating faster convergence of the \BnC algorithm).
Except for the ones mentioned above, other parameters of \Gurobi were set to their default values.
All computation was performed on a cluster of Intel(R) Xeon(R) Gold 6230R CPU @ 2.10 GHz computers.

In our experiments, we construct \CFLPLR instances using the procedure in \citet{Lin2021} \footnote{We use the code provided by  \citet{Lin2021} with  the random seed being $1$. The code is available at \url{https://www.researchgate.net/publication/348298635_GenerateRandomDatasetpy}.}, where the gravity rule \citep{Huff1964} is adopted to formulate the utilities.
The coordinates of customers and facilities are constructed from a 2-dimensional uniform distribution within the range $[0, 1000]^2$.
The buying power $b_i$ of customer $i$ is uniformly chosen from $\{10, 11, \dots, 1000\}$ and 
the fixed cost $f_j$ of facility $j$ is set to $2000$ for all $j \in[n]$.
The utilities of facilities to the customers are set to  {$u_{ij} = 1 / d_{ij}^2$} for all $i \in [m]$ and $\ j \in [n]$, where $d_{ij}$ is the Euclidean distance of customer $i$ and facility $j$, and the utilities of the outside options are set to $u_{i0}=1 / 50^2$ for all $i \in [m]$.
We construct two testsets of \CFLPLR instances based on different sizes of the problem and the consideration set $\gamma_i$, detailed as follows.

Testset \DATAONE consists of $32$ \CFLPLR instances where the number of customers $m$ and the number of facilities $n$ are taken from $\{800, 1000\}$ and $\{100, 200, 300, 400\}$, respectively. 
For the homogeneous case where $\gamma_i$, $i \in [m]$, are set to be the same $\gamma$, $\gamma$ is chosen from $\{1,2,3\}$; for the non-homogeneous (\nh) case, $\gamma_i$, $i \in [m]$, are uniformly chosen form $\{1, 2, 3, 4, 5\}$.

Testset \DATATWO consists of $30$ large-scale \CFLPLR instances where the customer behavior follows the partially binary rule, i.e., $\gamma_i =1$ for all $i \in [m]$.
In this testset, the number of customers $m$ and the number of facilities $n$ are taken from $\{1500, 2000, 3000, 5000, 10000\}$ and $\{100, 200, 300, 500, 1000, 2000\}$, respectively.

\subsection{Comparison of the \BnC algorithms based on formulations \eqref{prob:SC} and \eqref{prob:SC-lifted}}
\label{subsect:compu-submod-lifted}

\begin{table}[!htbp]
	\caption{Performance comparison of \testSI and \testLSI on the instances in testset \DATAONE.}
	\label{table:submodular-lifted}
	\footnotesize
	\scriptsize
	\centering
	\begin{tabular}{llrrrrrrrrrrrrrrrr}
	\toprule            
	\multicolumn{1}{l}{\multirow{2}{*}{\tblm}}                                                          
	&\multicolumn{1}{l}{\multirow{2}{*}{\tbln}}                                                         
	&\multicolumn{1}{l}{\multirow{2}{*}{\tblgamma}}                                                     
	& \multicolumn{7}{c}{\testSI}& \multicolumn{7}{c}{\testLSI} & \multicolumn{1}{c}{\multirow{2}{*}{\tblgapclosed}}\\
	\cmidrule(r){4-10}\cmidrule(r){11-17}
	&&&\multicolumn{1}{c}{\tblTg}&\multicolumn{1}{c}{\tblTC}&\multicolumn{1}{c}{\tblN}&\multicolumn{1}{c}{\tblZ}&\multicolumn{1}{c}{\tblZUB}&\multicolumn{1}{c}{\tblZLB}&\multicolumn{1}{c}{\tblTlp}
	&\multicolumn{1}{c}{\tblTg}&\multicolumn{1}{c}{\tblTC}&\multicolumn{1}{c}{\tblN}&\multicolumn{1}{c}{\tblZ}&\multicolumn{1}{c}{\tblZUB}&\multicolumn{1}{c}{\tblZLB}&\multicolumn{1}{c}{\tblTlp}
	&\\ \midrule
	
	800&100&1            &     9.6 &       0.1 &         1 &      105660 &      105663 &      105660 &       8.8 &     9.7 &       0.1 &         1 &      105660 &      105663 &      105660 &       8.9 &        --\\
	&&2                  &    17.7 &       3.8 &       753 &      130864 &      131713 &      128577 &      13.2 &    16.3 &       4.5 &         9 &      130864 &      131031 &      130669 &      13.9 &      80.3\\
	&&3                  &   157.2 &       8.1 &      5574 &      142928 &      145419 &      139246 &      18.5 &    23.9 &       9.4 &       284 &      142928 &      143397 &      142160 &      19.7 &      81.2\\
	&&\nh                &    34.6 &       5.0 &      1320 &      123790 &      124857 &      121643 &      17.0 &    23.0 &       8.4 &        30 &      123790 &      123911 &      123630 &      20.1 &      88.7\\
	&200&1            &    10.2 &       0.1 &         1 &      116529 &      116529 &      116529 &       9.3 &    10.1 &       0.1 &         1 &      116529 &      116529 &      116529 &       9.2 &        --\\
	&&2                  &    55.4 &      10.3 &      3079 &      141522 &      143311 &      138408 &      21.3 &    32.0 &      16.4 &         1 &      141522 &      141718 &      141368 &      27.0 &      89.1\\
	&&3                   &  (0.8) &      15.7 &     85362 &      152758 &      158253 &      145000 &      28.1 &   145.8 &      35.3 &      3336 &      152758 &      153973 &      149621 &      46.4 &      77.9\\
	&&\nh                 &  (0.1) &      13.1 &    240712 &      137447 &      140318 &      129131 &      27.2 &    74.6 &      33.4 &      1786 &      137447 &      138313 &      135989 &      46.1 &      69.8\\
	&300&1            &    10.5 &       0.1 &         1 &      124746 &      124746 &      124746 &       9.6 &    10.5 &       0.2 &         1 &      124746 &      124746 &      124746 &       9.7 &        --\\
	&&2                  &  3579.9 &      15.8 &    138592 &      148842 &      151945 &      143486 &      27.5 &    66.4 &      35.6 &      1238 &      148842 &      149461 &      143142 &      47.2 &      80.0\\
	&&3                   &  (1.8) &      28.7 &     36869 &      159364 &      166569 &      152540 &      42.7  &  (0.1) &      67.9 &    139200 &      159467 &      161808 &      154712 &      78.7 &      67.0\\
	&&\nh                 &  (0.2) &      23.7 &    199308 &      142398 &      145633 &      134305 &      39.4 &   129.1 &      90.1 &      1521 &      142398 &      143010 &      139547 &     102.6 &      81.1\\
	&400&1            &    11.0 &       0.2 &         1 &      130807 &      130807 &      130807 &      10.0 &    11.0 &       0.2 &         1 &      130807 &      130807 &      130807 &      10.1 &        --\\
	&&2                  &  3417.3 &      24.1 &    140442 &      153981 &      157109 &      143903 &      37.5 &    83.3 &      52.6 &       885 &      153981 &      154477 &      147847 &      64.8 &      84.2\\
	&&3                   &  (1.9) &      39.2 &     52146 &      164283 &      171349 &      151873 &      56.1  &  (0.1) &      97.7 &    123521 &      164508 &      166651 &      157308 &     107.5 &      68.7\\
	&&\nh                 &  (0.5) &      31.3 &    126873 &      145673 &      148992 &      138653 &      47.2 &   250.5 &     143.7 &      2794 &      145673 &      146495 &      142581 &     151.1 &      75.2\\
	1000&100&1           &     9.9 &       0.1 &         1 &      150422 &      150422 &      150422 &       9.0 &     9.8 &       0.1 &         1 &      150422 &      150422 &      150422 &       8.9 &        --\\
	&&2                  &    17.2 &       3.7 &       149 &      182178 &      182860 &      181136 &      13.6 &    16.4 &       4.8 &         1 &      182178 &      182264 &      181951 &      14.5 &      87.4\\
	&&3                  &    61.7 &       7.1 &      1342 &      196973 &      198928 &      194274 &      19.0 &    23.6 &       9.1 &        66 &      196973 &      197278 &      196344 &      20.3 &      84.4\\
	&&\nh                &    21.4 &       5.2 &       129 &      177850 &      178424 &      177818 &      17.1 &    19.1 &       6.7 &         1 &      177850 &      177850 &      177850 &      18.1 &     100.0\\
	&200&1           &    10.6 &       0.1 &         1 &      161637 &      161637 &      161637 &       9.6 &    10.6 &       0.1 &         1 &      161637 &      161637 &      161637 &       9.7 &        --\\
	&&2                  &   862.8 &      13.0 &     28860 &      193403 &      196718 &      184635 &      25.5 &    55.5 &      25.4 &       866 &      193403 &      194148 &      190494 &      37.0 &      77.5\\
	&&3                   &  (1.2) &      29.4 &     53637 &      207553 &      214829 &      194263 &      45.1 &  3666.6 &      39.4 &    104320 &      207553 &      209895 &      201833 &      52.5 &      67.8\\
	&&\nh                &   319.5 &      14.6 &      8943 &      193323 &      195802 &      188751 &      28.9 &    54.5 &      34.1 &       420 &      193323 &      193694 &      193105 &      48.2 &      85.1\\
	&300&1           &    11.4 &       0.2 &         1 &      173388 &      173388 &      173388 &      10.5 &    11.3 &       0.2 &         1 &      173388 &      173388 &      173388 &      10.4 &        --\\
	&&2                   &  (0.1) &      21.7 &    163025 &      202241 &      206005 &      193050 &      36.3 &    86.4 &      49.2 &      1213 &      202241 &      202978 &      198555 &      62.9 &      80.4\\
	&&3                   &  (1.8) &      47.7 &     28917 &      215695 &      223917 &      202934 &      69.4  &  (0.2) &      84.4 &     77637 &      215819 &      218704 &      205564 &      98.4 &      64.4\\
	&&\nh                 &  (0.6) &      27.6 &    116343 &      203408 &      208226 &      193975 &      45.9 &   740.0 &      84.0 &     30736 &      203408 &      204791 &      198810 &      94.9 &      71.3\\
	&400&1           &    13.7 &       0.2 &         1 &      171918 &      172007 &      171580 &      11.1 &    13.7 &       0.2 &         1 &      171918 &      172007 &      171580 &      11.0 &        --\\
	&&2                   &  (0.4) &      32.2 &     94607 &      202372 &      206967 &      194096 &      48.6 &   405.9 &      74.0 &     13571 &      202372 &      203815 &      194235 &      88.2 &      68.6\\
	&&3                   &  (2.8) &      75.2 &     16143 &      214652 &      225396 &      207888 &     101.8  &  (0.6) &     150.5 &     38449 &      215961 &      219917 &      207724 &     164.3 &      58.1\\
	&&\nh                 &  (0.5) &      25.3 &    146907 &      208478 &      212768 &      196018 &      43.3 &   410.0 &     160.4 &      8336 &      208478 &      209748 &      204764 &     170.4 &      70.4\\
	\tblSolved &&& 19& & & & & & & 28& & & & & & \\
	\ave&&               &   454.3 &             &     17325 &             &             &             &             &    26.4 &             &       200 &             &             &             &            \\ 
	\bottomrule         
\end{tabular}       
\end{table}

We first compare the performance of the proposed \BnC algorithms based on formulations \eqref{prob:SC} and \eqref{prob:SC-lifted}, where the submodular and lifted submodular inequalities are adopted, respectively.
For brevity, we refer to these two settings as \testSI and \testLSI, respectively.
\cref{table:submodular-lifted} summarizes the computational results on instances in testset \DATAONE.
For each setting, we report
the total CPU time in seconds (\tblT),
the CPU time in seconds spent in separating the cuts (\tblTC),
the number of explored nodes (\tblN),
the optimal value or the best incumbent of the instance (\tblZ), 
the upper bound (\tblZUB) and lower bound (\tblZLB) of the optimal value returned at the end of the first stage,  
and the CPU time in seconds spent in the first stage (\tblTlp).
For instances that cannot be solved to optimality within the given timelimit, we report under column \tblTg the percentage optimality gap (\tblG) computed as $\frac{\UB - \LB}{\UB} \times 100\%$, where $\UB$ and $\LB$ denote the upper bound and the lower bound obtained at the end of the time limit.
To compare the strength of the submodular and lifted submodular inequalities,  we also report the percentage gap improvement (\tblgapclosed), defined by
$\frac{\UB^1_{\texttt{LSI}}- \UB^1_{\texttt{SI}}}{\tblZ - \UB^1_{\texttt{SI}}} \times 100\%$, where
$\UB^1_{\texttt{LSI}}$ and $\UB^1_{\texttt{SI}}$ are the upper bounds at the end of the first stage returned by settings \testLSI and \testSI, respectively. 
The larger the \tblgapclosed, the more effective the lifted submodular inequalities over submodular  inequalities.
We use ``--'' to indicate that \testLSI and \testSI can return the same upper bounds, i.e., $\UB^1_{\texttt{LSI}}= \UB^1_{\texttt{SI}}$.
At the end of the table, we also report the total number of solved instances under each setting, and the average CPU time and number of explored nodes for the instances that can be solved to optimality by the two settings.

We first observe from  \cref{table:submodular-lifted} that for instances with $\gamma_i = 1$ for all $i \in [m]$ in testset \DATAONE, \testSI and \testLSI return the same upper bounds. 
This is reasonable as shown in \cref{thm:convex-hull-gamma1}, the submodular inequalities in \eqref{sc1-gamma1-unified} are able to describe the convex hull of $\X$, and thus the lifted submodular inequalities cannot provide additional improvement. 
For instances that $\gamma_i$ may take other values, the lifted submodular inequalities are much more effective than the submodular inequalities in terms of providing a stronger upper bound in the first stage.
Indeed, with the lifted submodular inequalities, \testLSI achieves a gap improvement \tblgapclosed ranging from $58.1\%$ to $100\%$ for these instances.
This is consistent with the result in \cref{cor:nontrivial-X}, where the lifted submodular inequalities are proven to be stronger than submodular inequalities.
In addition, the computational overhead spent in computing the lifted submodular inequalities is not large; see columns \tblTC under settings  \testSI and \testLSI.
Due to the above two advantages, \testLSI significantly outperforms \testSI, especially for instances with a large number of facilities $n$ or large sizes of the consideration sets $\gamma_i$.
In particular, equipped with the lifted submodular inequalities, \testLSI solves $28$ out of $32$ instances to optimality within the time limit,
while \testSI is only able to solve $19$ of them;
the average CPU time and number of explored nodes decrease from  $454.3$ seconds and $17325$ to $26.4$ seconds and $200$, respectively.

Due to the superior performance of \testLSI over \testSI,
in the following of this section, we only consider the proposed \BnC algorithm based on the formulation \eqref{prob:SC-lifted}.

\begin{table}[t]
	\caption{Performance comparison of \testGBD and \testLSI on the instances in testset \DATAONE.}
	\vspace*{-0.2cm}
	\label{table:GBD-lifted}
	\footnotesize
	\scriptsize
	\centering
	\begin{tabular}{llrrrrrrrrrrrrrrrr}
	\toprule            
	\multicolumn{1}{l}{\multirow{2}{*}{\tblm}}                                                          
	&\multicolumn{1}{l}{\multirow{2}{*}{\tbln}}                                                         
	&\multicolumn{1}{l}{\multirow{2}{*}{\tblgamma}}                                                     
	& \multicolumn{7}{c}{\testGBD}& \multicolumn{7}{c}{\testLSI} & \multicolumn{1}{c}{\multirow{2}{*}{\tblgapclosed}}\\
	\cmidrule(r){4-10}\cmidrule(r){11-17}
	&&&\multicolumn{1}{c}{\tblTg}&\multicolumn{1}{c}{\tblTC}&\multicolumn{1}{c}{\tblN}&\multicolumn{1}{c}{\tblZ}&\multicolumn{1}{c}{\tblZUB}&\multicolumn{1}{c}{\tblZLB}&\multicolumn{1}{c}{\tblTlp}
	&\multicolumn{1}{c}{\tblTg}&\multicolumn{1}{c}{\tblTC}&\multicolumn{1}{c}{\tblN}&\multicolumn{1}{c}{\tblZ}&\multicolumn{1}{c}{\tblZUB}&\multicolumn{1}{c}{\tblZLB}&\multicolumn{1}{c}{\tblTlp}
	&\\ \midrule
	
	800&100&1            &    19.5 &       1.1 &       706 &      105660 &      133574 &       77182 &      10.3 &     9.7 &       0.1 &         1 &      105660 &      105663 &      105660 &       8.9 &     100.0\\
	&&2                  &    33.1 &       1.7 &      1025 &      130864 &      152702 &      111107 &      11.1 &    16.3 &       4.5 &         9 &      130864 &      131031 &      130669 &      13.9 &      99.2\\
	&&3                  &   102.4 &       1.9 &      3977 &      142928 &      160116 &      129781 &      11.8 &    23.9 &       9.4 &       284 &      142928 &      143397 &      142160 &      19.7 &      97.3\\
	&&\nh                &    43.8 &       2.2 &       951 &      123790 &      141630 &       99566 &      14.0 &    23.0 &       8.4 &        30 &      123790 &      123911 &      123630 &      20.1 &      99.3\\
	&200&1            &    25.9 &       2.2 &       732 &      116529 &      158126 &       51725 &      11.9 &    10.1 &       0.1 &         1 &      116529 &      116529 &      116529 &       9.2 &     100.0\\
	&&2                  &   403.2 &       3.6 &     21796 &      141522 &      171353 &       99004 &      14.0 &    32.0 &      16.4 &         1 &      141522 &      141718 &      141368 &      27.0 &      99.3\\
	&&3                   &  (0.1) &       4.5 &    164624 &      152758 &      177931 &       98075 &      15.6 &   145.8 &      35.3 &      3336 &      152758 &      153973 &      149621 &      46.4 &      95.2\\
	&&\nh                &  1090.5 &       6.0 &     38630 &      137447 &      173490 &       68967 &      19.0 &    74.6 &      33.4 &      1786 &      137447 &      138313 &      135989 &      46.1 &      97.6\\
	&300&1            &    41.0 &       3.0 &      1149 &      124746 &      181155 &       14811 &      13.5 &    10.5 &       0.2 &         1 &      124746 &      124746 &      124746 &       9.7 &     100.0\\
	&&2                   &  (0.2) &       5.3 &    139934 &      148842 &      191106 &       61406 &      16.2 &    66.4 &      35.6 &      1238 &      148842 &      149461 &      143142 &      47.2 &      98.5\\
	&&3                   &  (0.7) &       7.2 &     26254 &      159438 &      196088 &       97956 &      18.9  &  (0.1) &      67.9 &    139200 &      159467 &      161808 &      154712 &      78.7 &      93.6\\
	&&\nh                 &  (0.1) &       7.5 &     50305 &      142398 &      184167 &       51602 &      20.5 &   129.1 &      90.1 &      1521 &      142398 &      143010 &      139547 &     102.6 &      98.5\\
	&400&1            &   136.5 &       3.8 &      5185 &      130807 &      199185 &       12634 &      14.7 &    11.0 &       0.2 &         1 &      130807 &      130807 &      130807 &      10.1 &     100.0\\
	&&2                   &  (0.3) &       8.6 &    104058 &      153981 &      205432 &       25743 &      20.0 &    83.3 &      52.6 &       885 &      153981 &      154477 &      147847 &      64.8 &      99.0\\
	&&3                   &  (0.9) &      11.9 &     22714 &      164508 &      208812 &       50616 &      24.6  &  (0.1) &      97.7 &    123521 &      164508 &      166651 &      157308 &     107.5 &      95.2\\
	&&\nh                 &  (0.2) &      11.5 &     86443 &      145673 &      196360 &       44008 &      25.1 &   250.5 &     143.7 &      2794 &      145673 &      146495 &      142581 &     151.1 &      98.4\\
	1000&100&1           &    14.9 &       1.4 &       118 &      150422 &      179702 &       94866 &      10.9 &     9.8 &       0.1 &         1 &      150422 &      150422 &      150422 &       8.9 &     100.0\\
	&&2                  &    35.8 &       1.8 &       877 &      182178 &      202753 &      141646 &      12.2 &    16.4 &       4.8 &         1 &      182178 &      182264 &      181951 &      14.5 &      99.6\\
	&&3                  &    57.9 &       2.7 &      1404 &      196973 &      213102 &      175664 &      13.9 &    23.6 &       9.1 &        66 &      196973 &      197278 &      196344 &      20.3 &      98.1\\
	&&\nh                &    36.0 &       2.4 &       610 &      177850 &      198824 &      153976 &      14.3 &    19.1 &       6.7 &         1 &      177850 &      177850 &      177850 &      18.1 &     100.0\\
	&200&1           &    31.7 &       2.9 &       718 &      161637 &      213344 &       40451 &      13.4 &    10.6 &       0.1 &         1 &      161637 &      161637 &      161637 &       9.7 &     100.0\\
	&&2                  &  2349.9 &       4.1 &     91985 &      193403 &      229742 &      110026 &      15.2 &    55.5 &      25.4 &       866 &      193403 &      194148 &      190494 &      37.0 &      98.0\\
	&&3                   &  (0.2) &       7.0 &    134337 &      207540 &      238117 &      142274 &      19.9 &  3666.6 &      39.4 &    104320 &      207553 &      209895 &      201833 &      52.5 &      92.3\\
	&&\nh                &   326.8 &       6.8 &      7770 &      193323 &      226953 &      148554 &      19.7 &    54.5 &      34.1 &       420 &      193323 &      193694 &      193105 &      48.2 &      98.9\\
	&300&1           &    39.8 &       3.9 &       678 &      173388 &      234877 &       29041 &      15.3 &    11.3 &       0.2 &         1 &      173388 &      173388 &      173388 &      10.4 &     100.0\\
	&&2                   &  (0.4) &       7.4 &     35196 &      202241 &      247945 &      116154 &      19.5 &    86.4 &      49.2 &      1213 &      202241 &      202978 &      198555 &      62.9 &      98.4\\
	&&3                   &  (0.8) &      10.2 &     23382 &      215819 &      254757 &      123216 &      25.0  &  (0.2) &      84.4 &     77637 &      215819 &      218704 &      205564 &      98.4 &      92.6\\
	&&\nh                 &  (0.3) &       6.8 &    114269 &      203408 &      249856 &      103270 &      20.3 &   740.0 &      84.0 &     30736 &      203408 &      204791 &      198810 &      94.9 &      97.0\\
	&400&1           &  6127.1 &       4.4 &    278190 &      171918 &      249473 &           0 &      16.4 &    13.7 &       0.2 &         1 &      171918 &      172007 &      171580 &      11.0 &      99.9\\
	&&2                   &  (0.7) &       8.9 &     40624 &      202372 &      259557 &       52808 &      22.1 &   405.9 &      74.0 &     13571 &      202372 &      203815 &      194235 &      88.2 &      97.5\\
	&&3                   &  (1.0) &      16.1 &     15343 &      215684 &      264704 &       78972 &      32.8  &  (0.6) &     150.5 &     38449 &      215961 &      219917 &      207724 &     164.3 &      91.9\\
	&&\nh                 &  (0.4) &      10.6 &     49219 &      208478 &      260710 &       58334 &      25.6 &   410.0 &     160.4 &      8336 &      208478 &      209748 &      204764 &     170.4 &      97.6\\
	\tblSolved &&& 18& & & & & & & 28& & & & & & \\
	\ave&&               &   606.4 &             &     25361 &             &             &             &             &    23.6 &             &       192 &             &             &             &            \\ 
	\bottomrule         
\end{tabular}       
	\vspace*{-0.2cm}
\end{table}

\begin{table}[!htbp]
	\caption{Performance comparison of \testGBD and \testLSI on the instances in testset \DATATWO.}
	\label{table:GBD-lifted-gamma1}
	\footnotesize
	\scriptsize
	\centering
	\begin{tabular}{llrrrrrrrrrrrrrrrr}
	\toprule            
	\multicolumn{1}{l}{\multirow{2}{*}{\tblm}}                                                          
	&\multicolumn{1}{l}{\multirow{2}{*}{\tbln}}                                                         
	&\multicolumn{1}{l}{\multirow{2}{*}{\tblgamma}}                                                     
	& \multicolumn{7}{c}{\testGBD}& \multicolumn{7}{c}{\testLSI} & \multicolumn{1}{c}{\multirow{2}{*}{\tblgapclosed}}\\
	\cmidrule(r){4-10}\cmidrule(r){11-17}
	&&&\multicolumn{1}{c}{\tblTg}&\multicolumn{1}{c}{\tblTC}&\multicolumn{1}{c}{\tblN}&\multicolumn{1}{c}{\tblZ}&\multicolumn{1}{c}{\tblZUB}&\multicolumn{1}{c}{\tblZLB}&\multicolumn{1}{c}{\tblTlp}
	&\multicolumn{1}{c}{\tblTg}&\multicolumn{1}{c}{\tblTC}&\multicolumn{1}{c}{\tblN}&\multicolumn{1}{c}{\tblZ}&\multicolumn{1}{c}{\tblZUB}&\multicolumn{1}{c}{\tblZLB}&\multicolumn{1}{c}{\tblTlp}
	&\\ \midrule
	
	1500&100&1           &    19.0 &       1.3 &       665 &      256723 &      288332 &      186464 &      11.1 &    10.2 &       0.2 &         1 &      256723 &      256723 &      256723 &       9.4 &     100.0\\
	&200&1           &    41.4 &       3.9 &      1026 &      290000 &      347989 &      147828 &      15.5 &    12.8 &       0.3 &         1 &      290000 &      290036 &      289877 &      11.0 &      99.9\\
	&300&1           &   131.7 &       6.4 &      3713 &      309503 &      378914 &      103080 &      20.1 &    13.2 &       0.3 &         1 &      309503 &      309513 &      309482 &      11.7 &     100.0\\
	&500&1            &  (0.1) &       9.6 &     59773 &      324397 &      428961 &       47682 &      28.8 &    17.0 &       0.6 &         1 &      324397 &      324397 &      324397 &      16.1 &     100.0\\
	&1000&1           &  (1.4) &      20.9 &     27759 &      331883 &      470641 &           0 &      51.3 &    23.9 &       0.9 &         1 &      333572 &      333572 &      333572 &      22.9 &     100.0\\
	&2000&1           &  (3.1) &      36.6 &     10909 &      331766 &      512399 &           0 &      77.6 &    42.3 &       2.0 &         7 &      337979 &      338075 &      337406 &      38.8 &      99.9\\
	2000&100&1           &    16.9 &       2.1 &        26 &      379652 &      406939 &      309425 &      13.4 &    11.6 &       0.3 &         1 &      379652 &      379652 &      379652 &      10.6 &     100.0\\
	&200&1           &    49.9 &       5.7 &       584 &      429073 &      495194 &      284227 &      21.2 &    13.0 &       0.2 &         1 &      429073 &      429073 &      429073 &      12.2 &     100.0\\
	&300&1           &   108.0 &       8.2 &      2362 &      451156 &      541602 &      255803 &      29.0 &    18.6 &       0.4 &         1 &      451156 &      451184 &      451156 &      16.9 &     100.0\\
	&500&1            &  (0.2) &      11.4 &     27219 &      470594 &      591178 &       59873 &      37.3 &    24.5 &       0.6 &         1 &      470594 &      470606 &      470579 &      22.7 &     100.0\\
	&1000&1           &  (2.1) &      26.5 &     19058 &      484333 &      641830 &           0 &      76.7 &    39.2 &       1.2 &         1 &      486795 &      486802 &      486795 &      38.1 &     100.0\\
	&2000&1           &  (17.1) &      46.2 &      3567 &      400678 &      701096 &           0 &     122.0 &    70.0 &       2.8 &         1 &      498922 &      498932 &      498343 &      67.6 &     100.0\\
	3000&100&1           &    20.0 &       2.8 &        22 &      620174 &      647358 &      583213 &      15.9 &    12.2 &       0.3 &         1 &      620174 &      620174 &      620174 &      11.2 &     100.0\\
	&200&1           &    55.4 &       7.0 &       589 &      718503 &      784453 &      555782 &      29.6 &    18.5 &       0.4 &         1 &      718503 &      718563 &      717947 &      16.6 &      99.9\\
	&300&1           &   106.0 &      11.8 &       702 &      763763 &      856222 &      549775 &      46.0 &    26.6 &       0.5 &         1 &      763763 &      763763 &      763763 &      25.6 &     100.0\\
	&500&1            &  (0.3) &      21.5 &     54091 &      796272 &      942141 &      293892 &      83.9 &    52.3 &       1.0 &         4 &      796272 &      796388 &      795251 &      46.8 &      99.9\\
	&1000&1           &  (1.4) &      37.7 &      5846 &      823187 &     1022094 &       61041 &     140.7 &   107.8 &       2.4 &         1 &      823192 &      823210 &      822322 &     104.9 &     100.0\\
	&2000&1           &  (12.6) &     102.1 &      2485 &      759660 &     1087212 &           0 &     291.5 &   185.2 &       4.4 &         3 &      841637 &      841720 &      836082 &     178.4 &     100.0\\
	5000&100&1           &    23.5 &       3.4 &         4 &     1151495 &     1172579 &     1132111 &      19.1 &    15.0 &       0.3 &         1 &     1151495 &     1151495 &     1151495 &      13.7 &     100.0\\
	&200&1           &    64.1 &       9.6 &       542 &     1354738 &     1421387 &     1250698 &      44.2 &    27.4 &       0.8 &         1 &     1354738 &     1354739 &     1354699 &      25.2 &     100.0\\
	&300&1           &   170.9 &      17.5 &       812 &     1457829 &     1570979 &     1201751 &      73.0 &    44.8 &       1.1 &         1 &     1457829 &     1457904 &     1457756 &      42.6 &      99.9\\
	&500&1            &  (0.2) &      31.0 &     69387 &     1514244 &     1674571 &      951090 &     127.5 &    97.1 &       1.6 &         1 &     1514244 &     1514351 &     1513888 &      94.2 &      99.9\\
	&1000&1           &  (1.8) &      77.0 &      9088 &     1564591 &     1810977 &      516730 &     501.8 &   464.4 &       3.2 &        12 &     1565000 &     1565180 &     1563879 &     455.3 &      99.9\\
	&2000&1           &  (21.2) &     173.7 &      2092 &     1542109 &     1903255 &       62060 &     738.5 &  2002.8 &       7.8 &         1 &     1590613 &     1590700 &     1589738 &    1991.3 &     100.0\\
	10000&100&1          &    40.3 &       7.0 &         2 &     2519882 &     2537414 &     2510408 &      35.7 &    23.2 &       1.1 &         1 &     2519882 &     2519882 &     2519882 &      21.2 &     100.0\\
	&200&1          &    90.8 &      13.2 &       343 &     2986742 &     3043403 &     2936880 &      75.5 &    48.0 &       1.3 &         1 &     2986742 &     2986742 &     2986742 &      46.0 &     100.0\\
	&300&1          &   247.1 &      24.8 &       569 &     3256111 &     3356070 &     3043955 &     180.0 &    95.1 &       2.1 &         1 &     3256111 &     3256111 &     3256111 &      93.3 &     100.0\\
	&500&1          &   878.2 &      60.8 &      2846 &     3429190 &     3611857 &     2928141 &     362.4 &   221.8 &       2.8 &         1 &     3429190 &     3429218 &     3429185 &     218.5 &     100.0\\
	&1000&1          &  (1.6) &     159.9 &      6484 &     3558234 &     3870442 &     1734264 &    1777.8 &  2370.2 &       7.0 &        10 &     3560001 &     3560364 &     3552374 &    2355.3 &      99.9\\
	&2000&1          &  (77.0) &     370.8 &       163 &      359693 &     4067646 &      359693 &    4094.9 &  6976.4 &      14.0 &       168 &     3633174 &     3633428 &     3617675 &    6953.9 &      99.9\\
	\tblSolved &&& 16& & & & & & & 30& & & & & & \\
	\ave&&               &   128.9 &             &       925 &             &             &             &             &    38.2 &             &         1 &             &             &             &            \\ 
	\bottomrule         
\end{tabular}       
\end{table}

\subsection{Comparison with the state-of-the-art \GBD approach in \cite{Lin2021}}
\label{subsect:lifted-GBD}

We now compare the performance of the proposed \BnC algorithm equipped with the lifted submodular inequalities with the state-of-the-art \GBD approach in \cite{Lin2021}, denoted by \testGBD.
For \testGBD, we also employ a two-stage implementation, as our preliminary experiments showed that it is more efficient than the single-stage implementation in \cite{Lin2021}.
The separation of Benders inequalities \eqref{ineq:GBD} is performed in a similar manner as in \cref{subsect:sepa}.
Detailed computational results on instances in testsets \DATAONE and \DATATWO under settings \testGBD and \testLSI are shown in \cref{table:GBD-lifted,table:GBD-lifted-gamma1}, respectively.
Similarly, to compare the effectiveness of the proposed lifted submodular inequalities and the Benders inequalities of \cite{Lin2021}, we report under column \tblgapclosed the percentage gap improvement, defined by
$\frac{\UB^1_\texttt{LSI}- \UB^1_\texttt{GBD}}{\tblZ - \UB^1_\texttt{GBD}} \times 100\%$, where
$\UB^1_\texttt{LSI}$ and $\UB^1_\texttt{GBD}$ are the upper bounds at the end of the first stage under settings \testLSI and \testGBD, respectively. 
In \cref{fig:GBD-vs-Lifted-profiles,fig:GBD-vs-Lifted-profiles-gamma1}, we further plot the performance profiles of the CPU time and number of explored nodes under settings \testGBD and \testLSI on instances in testsets \DATAONE and \DATATWO, respectively.

We first compare the performanace of settings \testGBD and \testLSI on instances in testset \DATAONE.
{From \cref{table:GBD-lifted}, we observe that while the separation of the lifted submodular inequalities takes more computational efforts compared to that of the Benders inequalities, the use of the lifted submodular inequalities  leads to a tighter formulation, resulting in a better upper bound.}
Indeed, with the lifted submodular inequalities, \testLSI consistently achieves a gap improvement \tblgapclosed of over $90\%$ for instances in testset \DATAONE. 
With a tighter formulation, the rounding heuristic in \testLSI exhibits better performance, which can be confirmed from \cref{table:GBD-lifted} that the lower bounds \tblZLB under \testLSI are much better than those under \GBD.
Due to the above advantages, \testLSI significantly outperforms  \testGBD.
Overall, \testLSI can solve $28$ out of all $32$ instances in testset \DATATWO, whereas \testGBD is only able to solve $18$ of them; 
for instances that can be solved by both settings, the average CPU time and average number of explored nodes under setting \testLSI are at least one magnitude smaller than those of \testGBD.
The computational efficiency of \testLSI is more intuitively depicted in \cref{fig:GBD-vs-Lifted,fig:GBD-vs-Lifted-node}, where we observe that approximately $80\%$ of the instances are solved by \testLSI within $1000$ seconds, whereas only half of the instances are solved by \testGBD within the same period of time;
more than half of the instances can be solved by \testLSI within $1000$ nodes, whereas approximately $25\%$ instances can be solved by \testGBD within the same number of nodes.

\begin{figure}[t]
	\centering
	\subcaptionbox{CPU time\label{fig:GBD-vs-Lifted}}{\includegraphics[scale=.40]{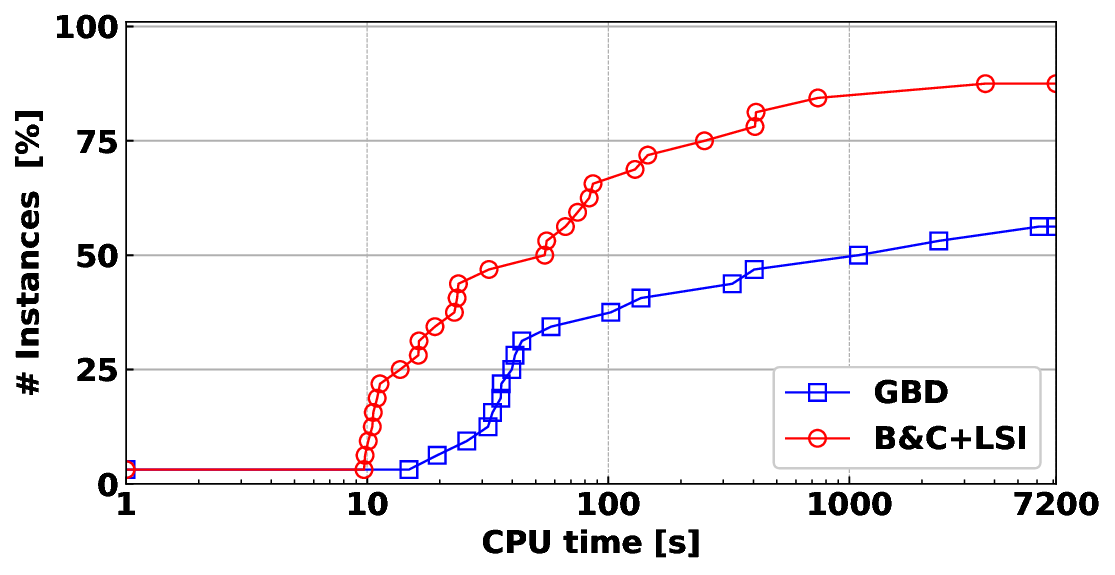}}
	\subcaptionbox{Number of explored nodes\label{fig:GBD-vs-Lifted-node}}{\includegraphics[scale=.40]{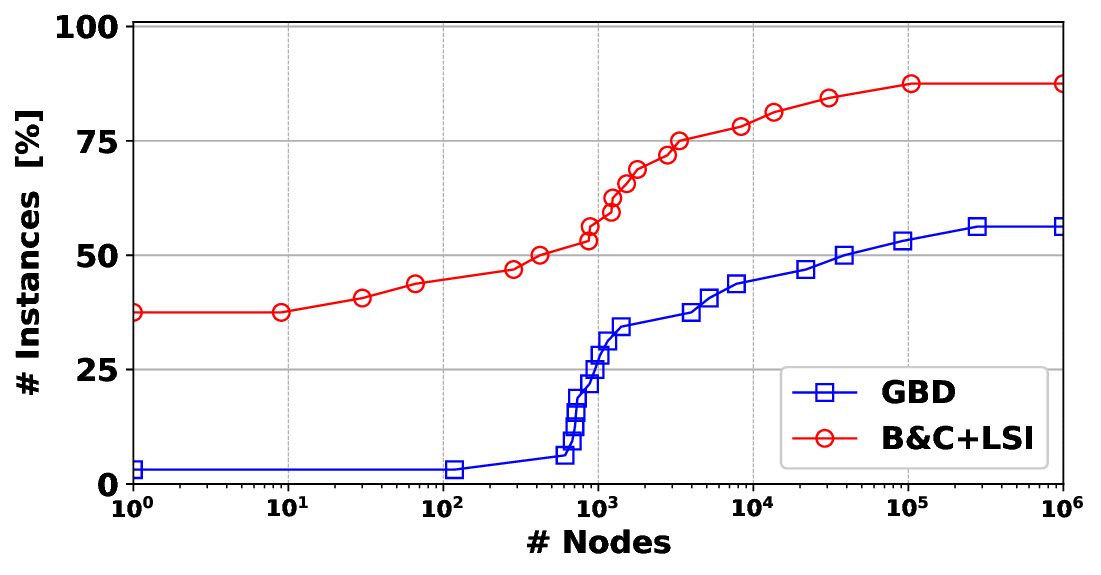}}
	\caption{Performance profiles of the CPU time and  number of explored nodes under settings \testGBD and \testLSI on instances in testset \DATAONE.}
	\label{fig:GBD-vs-Lifted-profiles}
\end{figure}

\begin{figure}[t]
	\centering
	\subcaptionbox{CPU time \label{fig:GBD-vs-Lifted-gamma1}}{\includegraphics[scale=.40]{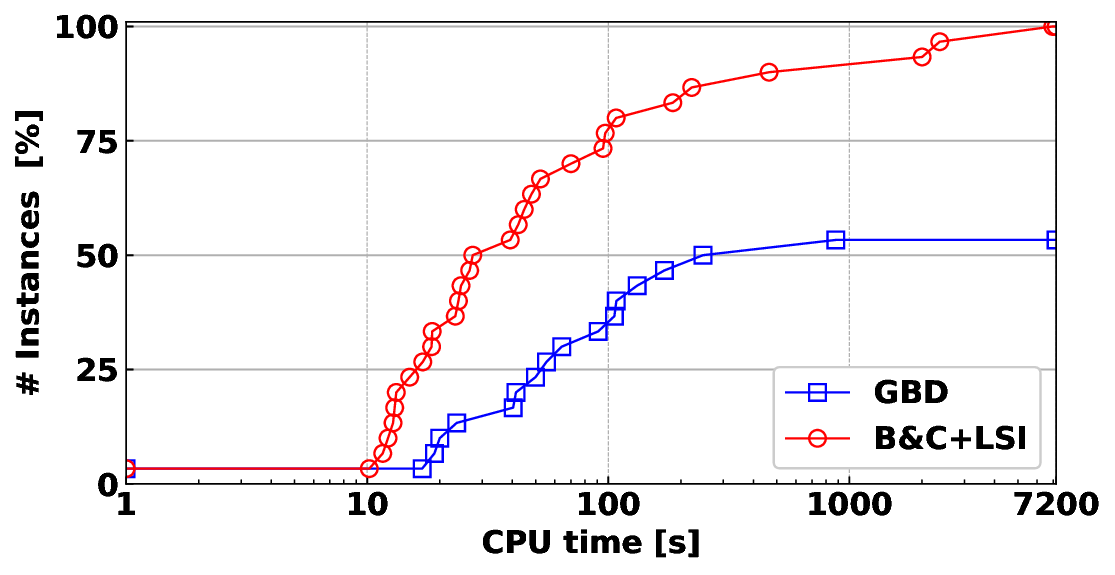}}
	\subcaptionbox{Number of explored nodes	\label{fig:GBD-vs-Lifted-gamma1-node}}{\includegraphics[scale=.40]{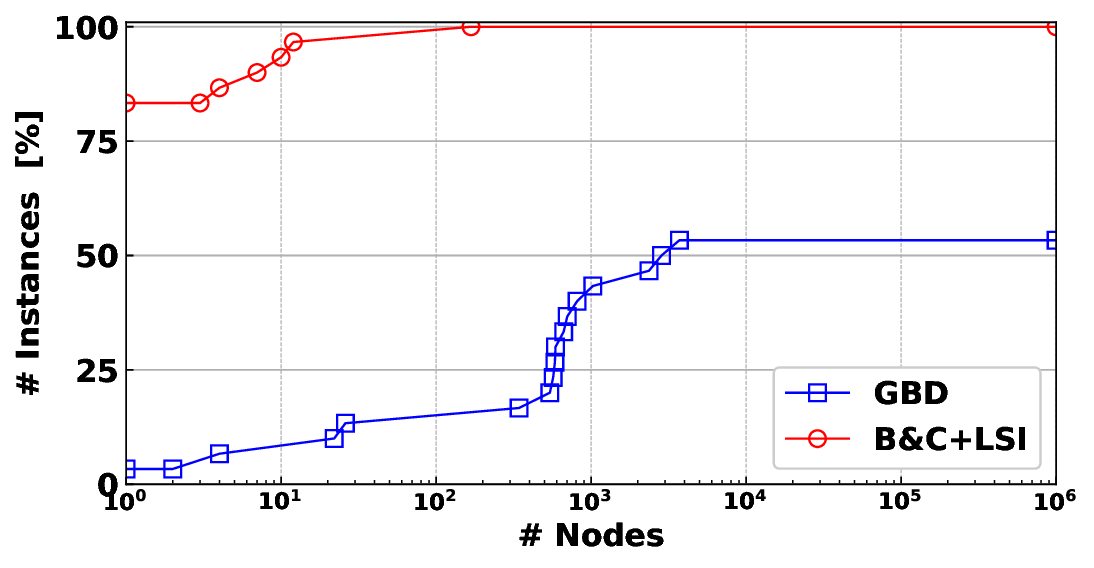}}
	\caption{Performance profiles of the CPU time and  number of explored nodes under settings \testGBD and \testLSI on instances in testset \DATATWO.}
	\label{fig:GBD-vs-Lifted-profiles-gamma1}
\end{figure}

We now compare the performance of settings \testGBD and \testLSI on the large-scale \CFLPLR instances with partially binary rule in testset \DATATWO. 
As discussed at the end of \cref{subsect:sepa}, for the case that $\gamma_i = 1$ for all $i \in [m]$, 
the separation of (lifted) submodular inequalities in  \eqref{sc1-gamma1-unified} admits an efficient $\mathcal{O}(n)$ exact algorithm. 
This is evident from \cref{table:GBD-lifted-gamma1} where the computational efforts required for the separation of the (lifted) submodular inequalities are relatively small. 
Moreover, as demonstrated in \cref{cor:nontrivial-X}, for the case where $\gamma_i =1$ for all $i \in [m]$, the (lifted) submodular inequalities are able to describe the convex hull of $\X$, and thus they enable to achieve a significantly better \LP relaxation; see column \tblgapclosed of \cref{table:GBD-lifted-gamma1}.
Therefore, \testLSI achieves a tremendously better performance than \testGBD.   
In particular, as shown in \cref{table:GBD-lifted-gamma1}, \testLSI can solve all instances to optimality,  whereas \testGBD fails to solve the instances with the number of facilities $n$ being $1000$ or $2000$;
as depicted in \cref{fig:GBD-vs-Lifted-gamma1,fig:GBD-vs-Lifted-gamma1-node}, the red circle lines {corresponding to} \testLSI are much ``higher'' than the blue square lines {corresponding to} \testGBD.

\section{Conclusions}\label{sect:Conclusion}

In this paper, we have investigated the polyhedral structure of the \CFLPLR and proposed an efficient \BnC approach for solving large-scale problems. 
Our approach first establishes the submodularity of the probability function characterizing a customer's probability to patronize new open facilities, and then uses the well-known submodular inequalities to describe the mixed 0-1 set defined by the hypograph of the probability function and the integrality constraints.
Moreover, we strengthen the submodular inequalities by sequential lifting, obtaining an \MILP formulation for the \CFLPLR described by the lifted submodular inequalities.  
Two key features of the proposed lifted submodular inequalities, which make them particularly well-suited to be embedded into a \BnC algorithm, are as follows.
First, their derivation explicitly takes the integrality constraints of the variables into consideration and they are guaranteed to define (strong) facet-defining inequalities of the convex hull of the considered mixed 0-1 set.
Second, their computation admits an efficient pseudo-polynomial time algorithm. 
Numerical results showed that compared with the Benders inequalities in the state-of-the-art \GBD algorithm whose derivation does not take the integrality constraints into consideration, the proposed lifted submodular inequalities are much more effective in terms of providing a tight \LP relaxation bound. 
With this advantage, it is shown that the proposed \BnC algorithm outperforms the state-of-the-art approach by at least one order of magnitude.

\bibliography{shorttitles,cflp}
\newpage

\appendix

%

\section{Proof of Proposition \ref{prop:SC-simplification}} 

\label{sect:appendix-SC-simplification}
\begin{proof}
	It suffices to consider the case $|\CS| \ge \gamma + 1$. 
	In this case, it is easy to see  
	\begin{align*}
		& \Phi(\CS) = \Phi(\CS^{\gamma}),\\
		& \rho_j(\CS) = \rho_j(\CS^{\gamma}), \ \forall \ j \in [n] \backslash \CS, \\
		& \rho_j([n]\backslash j) = 0,~ \rho_j(\CS^{\gamma}) = 0, \ \forall \ j \in \CS \backslash \CS^{\gamma}. 
	\end{align*}
	Thus, 
		\begin{equation*} 
			\begin{aligned}
				& \Phi(\CS) + \sum_{j \in [n] \backslash \CS}\rho_j(\CS) x_j 
				- \sum_{j \in \CS} \rho_j([n]\backslash j) (1-x_j)\\
				=~ & \Phi(\CS) + \sum_{j \in [n] \backslash \CS} \rho_j(\CS) x_j - \sum_{j \in \CS^{\gamma}} \rho_j([n] \backslash j)(1-x_j)-\underbrace{\sum_{j \in \CS \backslash \CS^{\gamma}} \rho_j([n] \backslash j)(1-x_j)}_{=0} \\
				= ~& \Phi(\CS) + \sum_{j \in [n] \backslash \CS^\gamma} \rho_j(\CS^{\gamma}) x_j 
				- \underbrace{\sum_{j \in \CS\backslash\CS^{\gamma}} \rho_j(\CS^{\gamma}) x_j}_{=0} - \sum_{j \in \CS^{\gamma}} \rho_j([n] \backslash j)(1-x_j) \\
				= ~& \Phi(\CS^{\gamma}) + \sum_{j \in [n] \backslash \CS^{\gamma}} \rho_j(\CS^{\gamma}) x_j
				- \sum_{j \in \CS^{\gamma}} \rho_j([n] \backslash j)(1-x_j).
			\end{aligned}
	\end{equation*}
	This shows statement (i).
	
	Similarly, it is simple to check 
	\begin{subequations}
		\begin{align*}
			& \Phi(\CS) = \Phi(\bar{\CS}^\gamma),\\
			& \rho_j(\CS \backslash j) = \rho_j(\bar{\CS}^\gamma \backslash j), \ \forall  \ j \in \CS,\\
			& \rho_j(\bar{\CS}^\gamma \backslash j) = 0, \ \forall \ j \in \bar{\CS}^\gamma \backslash \CS. 
		\end{align*}
	\end{subequations}
	Hence,
	\begin{equation*}
		\begin{aligned}
			& \Phi(\bar{\CS}^\gamma) - \sum_{j \in \bar{\CS}^\gamma} \rho_j(\bar{\CS}^\gamma \backslash j) (1 - x_j)
			+ \sum_{j \in [n] \backslash\bar{\CS}^\gamma} \rho_j(\varnothing) x_j\\
			= ~&  \Phi(\bar{\CS}^\gamma) - \sum_{j \in\CS} \rho_j(\bar{\CS}^\gamma \backslash j) (1 - x_j) -  \underbrace{\sum_{j \in\bar{\CS}^\gamma\backslash \CS} \rho_j(\bar{\CS}^\gamma \backslash j) (1 - x_j)}_{=0}
			+ \sum_{j \in [n] \backslash \bar{\CS}^\gamma} \rho_j(\varnothing) x_j\\
			= ~& \Phi(\CS) - \sum_{j \in \CS} \rho_j(\CS \backslash j) (1 - x_j)
			+ \sum_{j \in [n] \backslash \bar{\CS}^\gamma} \rho_j(\varnothing) x_j\\
			\le ~& \Phi(\CS) - \sum_{j \in \CS} \rho_j(\CS \backslash j) (1 - x_j)
			+ \sum_{j \in [n] \backslash \CS} \rho_j(\varnothing) x_j,
		\end{aligned}
	\end{equation*}
	where the last inequality follows from $\rho_j(\varnothing) \ge 0$, $x_j \in \{0,1\}$, and $\CS \subseteq  \bar{\CS}^\gamma$.	
	Thus, statement (ii) also holds true.
\end{proof}
\section{Proof of Theorem \ref{thm:convex-hull-gamma1}} 
\label{sect:appendix-convex-hull-gamma1}

To prove Theorem \ref{thm:convex-hull-gamma1},
we need the following lemma. 
\begin{lemma}\label{obs:gamma1-minRHS}
	Let $f_{\ell}(x)$ be the right-hand sides of inequalities (24), 
	i.e., 
	\begin{equation}\label{fdef}\small
		\begin{aligned}
			\RHS_{\ell}(x) 
			=  \frac{ u_{\ell+1}}{u_{\ell+1} + u_0} + \sum_{j = 1}^{n}\left( \frac{u_j}{u_j + u_0} - \frac{u_{\ell+1}}{u_{\ell+1} + u_0}\right)^+ x_j
			=  \frac{ u_{\ell+1}}{u_{\ell+1} + u_0} + \sum_{j = 1}^{\ell }\left( \frac{u_j}{u_j + u_0} - \frac{u_{\ell+1}}{u_{\ell+1} + u_0}\right) x_j, \ \forall~\ell \in [n].
		\end{aligned}
	\end{equation}
	For any $x^\ast \in [0,1]^n$, it follows 
	\begin{equation}\label{goal-lem}
		f_{k(x^\ast)}(x^\ast) = \min \left\{f_{\ell}(x^\ast) \, : \, \ell \in [n]\right\},
	\end{equation}
	where 
	\begin{equation}\label{kdef}
	k(x^*) =\left\{
	\begin{array}{ll}
		1, & ~\text{if}~x_1^*=1;\\
		\max \left\{ \ell \in [n]: \sum_{j = 1}^{\ell } x^\ast _j < 1\right\}, & \text{otherwise}.
	\end{array}\right.
	\end{equation}
\end{lemma}
\begin{proof}
	Given any $\ell \in \{2, 3, \ldots, n\}$, it follows
	\begin{equation*}
		\begin{aligned}
			\RHS_{\ell}(x^\ast) - \RHS_{\ell-1}(x^\ast)
			& = \frac{u_{\ell+1}}{u_{\ell+1} + u_0}
			+ \sum_{j = 1}^{\ell}\left( \frac{u_j}{u_j + u_0} - \frac{u_{\ell+1}}{u_{\ell+1} + u_0} \right) x^\ast_j
			- \frac{u_{\ell}}{u_{\ell} + u_0}
			- \sum_{j = 1}^{\ell - 1}\left( \frac{u_j}{u_j + u_0} - \frac{u_{\ell }}{u_{\ell } + u_0} \right)x^\ast_j\\
			& = \left(\frac{u_{\ell}}{u_{\ell} + u_0} - \frac{u_{\ell+1}}{u_{\ell+1} + u_0}\right) \cdot \left( \sum_{j=1}^{\ell} x^\ast_j - 1\right).
		\end{aligned}
	\end{equation*}	
	Together with the definition of $k(x^\ast)$, we have $f_1(x^\ast) \ge f_2(x^\ast) \ge \cdots \ge f_{k(x^*)}(x^\ast)$ and 
	$f_{k(x^*)}(x^\ast) \le f_{k(x^*)+1}(x^\ast) \le \cdots \le f_{n}(x^\ast)$. 
	Thus equation \eqref{goal-lem} holds true.
\end{proof}

\begin{proof}[Proof of Theorem \ref{thm:convex-hull-gamma1}]
	It suffices to show that every extreme point $(w^\ast, x^\ast)$ of $\CP:= \{(w, x) \in \R \times [0,1]^n: w \le \RHS_{\ell}(x), ~\forall~\ell \in [n]\}$ is integral.
	Suppose, otherwise, that $(w^\ast, x^\ast)$ is  a fractional extreme point of $\CP$, i.e., $x^\ast_j \in (0,1)$ for some $j \in [n]$.
	We have the following two cases.
	\begin{itemize}
		\item[(i)] $x^\ast_t \in (0,1)$ holds for some $t \in \{k(x^\ast)+1, k(x^\ast)+2, \dots, n\}$. 
		Let $(w^1, x^1) = (w^\ast, x^\ast + \delta \boldsymbol{e}_t)$
		and $(w^2, x^2) = (w^\ast, x^\ast - \delta \boldsymbol{e}_t)$, where 
		$\delta >0$ is a sufficiently small value.
		By the definition of $f$ in \eqref{fdef}, we have $ f_{k(x^*)}(x^*)=f_{k(x^*)}(x^1)$.
		From $t \geq k(x^*)+1$ and the definitions of $x^1$ and $k(x^*)$ in \eqref{kdef},  it follows $k(x^*)=k(x^1)$. 
		Thus, $w^1 =w^* \leq f_{k(x^*)}(x^*) = f_{k(x^*)}(x^1)= f_{k(x^1)}(x^1)$.
		From Lemma B.1, this implies $w^1  \leq f_{\ell}(x^1)$ for all $\ell\in[n]$, and therefore $(w^1, x^1) \in \CP$. 
		The proof of $(w^2, x^2) \in \CP$ is similar.
		Since $(w^\ast, x^\ast) = \frac{1}{2} (w^1, x^1) + \frac{1}{2} (w^2, x^2)$ and $\delta >0$, 
		$(w^\ast, x^\ast)$ cannot be an extreme point of $\CP$, a contradiction. 
		\item[(ii)] $x^\ast_{t} \in (0,1)$ holds for some $t \in [k(x^*)]$ and $x^*_\ell \in \{0,1\}$ holds for all $\ell \in \{k(x^\ast)+1,k(x^\ast)+2,\ldots, n\}$.
		Then, 
		from the definition of $k(x^*)$ in \eqref{kdef}, this implies 
		\begin{equation}\label{def1}
			x_{k(x^\ast)+1}^*=1, ~x_1^*\neq 1, ~\text{and}~ \sum_{j=1}^{k(x^*)} x_j^* < 1.
		\end{equation}
		Let $(w^3, x^3) = (w^\ast + \zeta, x^\ast + \delta \boldsymbol{e}_{t})$
		and $(w^4, x^4) = (w^\ast - \zeta, x^\ast - \delta \boldsymbol{e}_{t})$, where 
		$\delta >0$ is a sufficiently small value and $\zeta = (\frac{u_{t}}{u_{t} + u_0} - \frac{u_{k(x^*)+1}}{u_{k(x^*)+1} + u_0}) \cdot \delta$. 
		As $\delta$ is sufficiently small,  by \eqref{def1}, we have $k(x^3)=k(x^4)=k(x^*)$, and thus
		\begin{align}
			&  \RHS_{k(x^3)}(x^3)= \RHS_{k(x^*)}(x^3)= \RHS_{k(x^*)}(x^\ast) + \left(\frac{u_{t}}{u_{t} + u_0} - \frac{u_{k(x^*)+1}}{u_{k(x^*)+1} + u_0}\right) \cdot \delta= \RHS_{k(x^*)}(x^\ast) + \zeta,
			\label{ineq-tmp22}
			\\
			&  \RHS_{k(x^4)}(x^4) =\RHS_{k(x^*)}(x^4)= \RHS_{k(x^*)}(x^\ast) - \left(\frac{u_{t}}{u_{t} + u_0} - \frac{u_{k(x^*)+1}}{u_{k(x^*)+1}+ u_0}\right) \cdot \delta=\RHS_{k(x^*)}(x^\ast) - \zeta.
			\label{ineq-tmp23}
		\end{align}
		From $w^3=w^*+\zeta$, $w^4=w^*-\zeta$, $w^* \leq f_{k(x^*)}(x^*)$, \eqref{ineq-tmp22}, and \eqref{ineq-tmp23}, it follows $w^3 \leq \RHS_{k(x^3)}(x^3)$ and $w^4\leq \RHS_{k(x^4)}(x^4)$.
		Therefore, using Lemma B.1, this implies $w^3  \leq f_{\ell}(x^3)$ and $w^4  \leq f_{\ell}(x^4)$  for all $\ell\in[n]$, and thus $(w^3, x^3),  (w^4, x^4)\in \CP$. 
		Since $(w^\ast, x^\ast) = \frac{1}{2} (w^3, x^3) + \frac{1}{2} (w^4, x^4)$ and $\delta > 0$, point $(w^\ast, x^\ast)$ cannot be an extreme point of $\CP$, a contradiction. 
		 \qedhere
	\end{itemize}
\end{proof}

\section{Proof of Theorem \ref{thm:np-hard}}
\label{sect:appendix-nphard}
	We prove that problem \eqref{prob:lift-NPhard}
	is as hard as the partition problem, which is {NP-complete} \citep{Garey1978}.
	First, we introduce the partition problem: 
	given a finite set $\CN = \{1, 2, \dots, p\}$ of $p$ elements and a size $a_i \in \Z_+$ for the $i$-th element with $\sum_{i \in \CN} a_i = 2b$, does there exist a partition $\CN = \CN_1 \cup \CN_2$ such that $\sum_{i \in \CN_1} a_i = \sum_{i \in \CN_2} a_i = b$?
	Without loss of generality, we assume $b >0$.
	
	Given any instance of a partition problem, we construct an instance of problem \eqref{prob:lift-NPhard}
	by setting $q=\gamma= p$, $u_0 = b$, $u_j = a_j$, and $\alpha_j = a_j / 4b$ for $j \in [p]$.
	As $\gamma = p$, the cardinality constraint $\sum_{j = 1}^p y_j \le \gamma$ in problem \eqref{prob:lift-NPhard}
	becomes redundant.
	Thus, problem \eqref{prob:lift-NPhard}
	reduces to
	\begin{equation}\label{prob-complexity-liftprobg}
		\begin{aligned}
			\nu := \max\left\{
			 \frac{\sum_{j = 1}^p a_j y_j }{\sum_{j = 1}^p a_j y_j + b} - \frac{1}{4b} \sum_{j = 1}^p a_j y_j: y \in \{0,1\}^p
			\right\}.
		\end{aligned}
	\end{equation}
	Letting $f(z) = \frac{z}{z+b} - \frac{z}{4b}$ where $z \in [0,2b]$, it is simple to check that $f(z) \leq 1/4$ where the only maximum point is arrived at $z=b$. 
	Using this fact, $\nu = 1/4$ holds if and only if $\sum_{j = 1}^p a_j y_j = b$ holds for some $y \in \{0,1\}^p$, which is further equivalent to the existence of the partition problem $\CN = \CN_1 \cup \CN_2$ with $\sum_{i \in \CN_1} a_i = \sum_{i \in \CN_2} a_i = b$, that is, the answer to the partition problem is yes. 
	Since the above transformation can be done in polynomial time and 
	the partition problem is NP-complete, we conclude that problem \eqref{prob-complexity-liftprobg} is NP-hard.

\end{document}